\documentclass[12pt,leqno]{amsart}
\usepackage{amsaddr}
\usepackage[margin=1in,footskip=0.25in]{geometry}
\usepackage[latin9]{inputenc}
\usepackage{verbatim}
\usepackage{amsthm}
\usepackage{amstext}
\usepackage{amsmath,amssymb,graphicx,upgreek,appendix}
\usepackage{amsrefs}
\usepackage{setspace,subfigure,psfrag,subfigure}
\usepackage[vlined,boxruled,longend]{algorithm2e}
\SetKwInOut{Input}{input}
\SetKwInOut{Output}{output}
\SetKw{KwDownto}{down to}
\IncMargin{1.2em}	
\LinesNumbered
\usepackage{latexsym}
\usepackage{mathrsfs}
\usepackage{cite}
\usepackage{amssymb}
\usepackage{amsthm}
\usepackage{amsmath}
\usepackage{tikz-cd}
\usepackage{url}
\usepackage[section]{placeins} 
 \usepackage[pagewise]{lineno}
\numberwithin{equation}{section}
\numberwithin{figure}{section}

\newtheorem{definition}{Definition}
\newtheorem{fact}{Fact}

\newtheorem{theorem}{Theorem}

\newtheorem{corollary}[theorem]{Corollary}
\newtheorem{proposition}{Proposition}

\newtheorem{remark}{Remark}
\newtheorem*{theorem*}{Theorem}
\newtheorem*{cor*}{Corollary}
\newtheorem*{prop*}{Proposition}
\theoremstyle{definition}
\newtheorem{example}{Example}
\makeatletter
\newtheorem*{rep@theorem}{\rep@title}
\newcommand{\newreptheorem}[2]{%
\newenvironment{rep#1}[1]{%
 \def\rep@title{#2 \ref{##1}}%
 \begin{rep@theorem}}%
 {\end{rep@theorem}}}
\makeatother

\newreptheorem{theorem}{Theorem}
\newreptheorem{prop}{Proposition}
\newreptheorem{cor}{Corollary}

\newcommand{\such}{\text{ such that }}
\newcommand{\eqdef}{\triangleq}

\newcommand{\cF}{\mathcal{F}}

\newcommand{\mathset}[1]{\left\{#1\right\}}

\newcommand{\abs}[1]{\left|#1\right|}

\newcommand{\parenv}[1]{\left( #1 \right)}
\newcommand{\sparenv}[1]{\left[ #1 \right]}

\newcommand{\tends}[1]{\operatorname*{\longrightarrow}\limits_{#1}}
\newcommand{\norm}[1]{\left\lVert#1\right\rVert}

\DeclareMathOperator{\fin}{Fin}

\DeclareMathOperator{\topent}{h}

\DeclareMathOperator{\PM}{PM}
\DeclareMathOperator{\PC}{PC}
\DeclareMathOperator{\ima}{Img}
\DeclareMathOperator{\Int}{Int}
\DeclareMathOperator{\fix}{Fix}
\DeclareMathOperator{\Is}{Iso}

\DeclareMathOperator{\End}{End}

\renewcommand{\Bbb}{\mathbb}

\usepackage{xcolor}
\newcommand{\C}{{\Bbb C}}
\newcommand{\N}{{\Bbb N}}
\newcommand{\R}{{\Bbb R}}
\newcommand{\Z}{{\Bbb Z}}

\newcommand{\E}{{\Bbb E}}

\newcommand*{\fullref}[1]{\hyperref[{#1}]{\autoref*{#1}}}

\begin{document}

\title[Permutations With Restricted Movement]{Permutations With Restricted Movement}

\author{Dor Elimelech}
\address{School of Electrical and Computer Engineering, Ben-Gurion University of the Negev, Israel}
\email{doreli@post.bgu.ac.il}



\date{\today}

\thanks{}

\maketitle

\begin{center}
	\textbf{Abstract}
\end{center}

A restricted permutation of a locally finite directed graph $G=(V,E)$ is a vertex permutation $\pi: V\to V$ for which $(v,\pi(v))\in E$, for any vertex $v\in V$. The set of such permutations, denoted by $\Omega(G)$, with a group action induced from a subset of graph isomorphisms form a topological dynamical system. We focus on the particular case presented by Schmidt and Strasser \cite{SchStr17} of restricted $\Z^d$ permutations, in which $\Omega(G)$ is a subshift of finite type.
We show a correspondence between restricted permutations and perfect matchings (also known as dimer coverings). We use this correspondence in order to investigate and compute the topological entropy in a class of cases of restricted $\Z^d$-permutations. We discuss the global and local admissibility of patterns, in the context of restricted $\Z^d$-permutations. Finally, we review the related models of injective and surjective restricted functions.


\section{Introduction}
\label{CHA:introduction}

	In this paper, we study topological dynamical systems defined by permutations of locally finite graphs with restricted movement. We generalize the model suggested by Schmidt and Strasser \cite{SchStr17} of $\Z^d$ permutations with restricted movement to general locally finite graphs. We find a natural correspondence of such permutations with perfect matchings and use it in order compute the exact entropy for a class of examples.

Let $G=(V,E)$ be a directed graph. A permutation of $V$, $\pi\in S(V)$, is said to be restricted by $G$ if for all $v\in V$, $(v,\pi(v))\in E$.
We define $\Omega(G)$ to be the set of all permutations restricted by $G$. If $G$ is locally finite, equipped with discrete topology, and $\Gamma$ is a group acting on $G$ by graph isomorphisms, then the action of $\Gamma$ on $G$ extends naturally to a continuous action on $\Omega(G)$ (which is a compact space with the product topology).

Schmidt and Strasser \cite{SchStr17} investigated SFTs defined by permutations of $\Z^d$ with movements restricted by some finite set $A\subseteq \Z^d$. That is, permutations of $\Z^d$ satisfying $\pi(n)-n\in A$ for all $n\in \Z^d$.  In the terminology of our model, those are permutations restricted by the graph $G_A=(\Z^d,E_A)$, where 
\[ E_A\eqdef \mathset{(n,m)\in \Z^d\times \Z^d : m-n\in A}.\] 
An element $n\in \Z^d$ acts on permutations by the conjugation $\sigma_n \circ \pi \circ \sigma_n^{-1}$, where $\sigma_n$ is the shift by $n$ operation on $\Z^d$ defined by $\sigma_n(m)\eqdef m+n$. We use the notation
\[ \Omega(A) \eqdef \Omega(G_A)= \mathset{\pi\in S(\Z^d) : \forall n\in \Z^d, \pi(n)-n\in A }, \]
where $S(\Z^d)$ denotes the set of all permutations of $\Z^d$.

In Section~\ref{CHA:PM} we show a correspondence between restricted permutations and perfect matchings. A perfect matching of an undirected graph $G=(V,E)$ is a subset of edges $M\subseteq E$ such that any vertex is covered by a unique edge from $M$. We denote the set of perfect matchings of $G$ by $\PM(G)$. In Theorem~\ref{th:GenCan} we show a general characterization of restricted movement permutations by perfect matchings.

\begin{reptheorem}{th:GenCan}
There is a homeomorphism, $\Psi$, between the elements of $\Omega(G)$ and the perfect matchings in  $\PM(G')$.
If  a group $\Gamma$ acts on $G$ by graph isomorphisms, then the action of $\Gamma$ on $G'$ induces a group action of $\Gamma$ on $\PM(G')$ such that $(\Omega(G),\Gamma)$ and $(\PM(G'),\Gamma)$ are topologically conjugated and $\Psi$ is a conjugation map. That is, $\Psi$ is an homeomorphism and the following diagram commutes
\[
\begin{tikzcd}
\Omega(G) \arrow[r,"\Gamma"] \arrow[d,swap,"\Psi"] &
 \Omega(G) \arrow[d,"\Psi"] \\
\PM(G')  \arrow[r,"\Gamma"] & \PM(G') 
\end{tikzcd}
\]
\end{reptheorem}

By this theorem, we may identify $\Z^2$-permutations with movements restricted by the set $A_L\eqdef \mathset{(0,0),(0,1),(1,0)}$ and perfect matching of the well known honeycomb lattice. We use the theory of perfect matchings of $\Z^2$-periodic bipartite planar graphs \cite{KenOkoShe06} in order to derive an exact expression for the topological entropy of $\Omega(A_L)$, and for the periodic entropy, which is the logarithmic growth-rate of the number of its periodic points. We denote the topological entropy and the periodic entropy  by $\topent(\Omega(A_L))$ and $\topent_p(\Omega(A_L))$ respectively.

\begin{reptheorem}{th:SolL} For $A_L=\mathset{(0,0),(0,1),(1,0)}$,
\[ \lim_{n\to\infty}\frac{\log\abs{\fix_{n\Z^2}(\Omega(A_L))}}{n^2}  \eqdef \topent_p(\Omega(A_L))=\topent(\Omega(A_L))=\frac{1}{4\pi^2}\intop_0^{2\pi}\intop_0^{2\pi}\log\abs{1+e^{ix}+e^{iy}}dxdy,\]
where $\fix_{n\Z^2}(\Omega(A))$ denotes the set of $n$-periodic points in $\Omega(A_L)$.
\end{reptheorem}
In this theorem, we also provide a solution for Problem 6.3.1 in \cite{SchStr17}. For the benefit of proving the existence of a polynomial-time algorithm for counting the rectangular patterns in the SFT $\Omega(A_L)$, we establish the concept of perfect covers. We describe a polynomial-time algorithm for counting perfect covers (Proposition~\ref{prop:GenAlg}), which we use for counting rectangular patterns in $\Omega(A_L)$ (Proposition~\ref{prob:algA_l}).

In Theorem~\ref{th:PerToPM} we describe a different correspondence between permutations and perfect matchings which holds for bipartite symmetric graphs.
\begin{reptheorem}{th:PerToPM}
Let $G=(V\uplus U,E)$ be a directed bipartite graph. There is a continuous embedding of $\Omega(G)$ inside $\PM(\tilde{G})\times \PM(\tilde{G})$, where $\tilde{G}=(V\uplus U,\tilde{E})$ is the undirected version of $G$.
If a group $\Gamma$ acts on $G$ by bipartite graph isomorphisms (that is, $\gamma V=V$ for all $\gamma\in \Gamma$) then it induces an action on $\PM(\tilde{G})$ and the following diagram commutes: 
\[
\begin{tikzcd}
\Omega(G) \arrow[r,"\Gamma"] \arrow[d,swap,"\Phi"] &
 \Omega(G) \arrow[d,"\Phi"] \\
\PM(\tilde{G})^2  \arrow[r,"\Gamma"] & \PM(\tilde{G})^2 
\end{tikzcd}
\]
Furthermore, if $G$ is symmetric (that is, $(v,u)\in E$ implies $ (u,v)\in E$), $\Phi$ is a homeomorphism and $(\Omega(G),\Gamma)$ and $(\PM(\tilde{G}\times \tilde{G}),\Gamma) $ are topologically conjugated. 
\end{reptheorem}
We similarly use Theorem~\ref{th:PerToPM} in order to investigate the SFT $\Omega(A_+)$, where $A_+\eqdef\mathset{(\pm 1,0),(0,\pm 1)}$. Combined with Kasteleyn's results concerning perfect matchings on the two-dimensional square lattice \cite{Kas61}, we prove the following:
\begin{reptheorem}{th:Sol+} The $\mathbb{Z}^2$-SFT $\Omega(A_+)$ is isomorphic to the double-dimer model on the square lattice $\mathbb{Z}^2$, and as a consequence, its topological entropy and periodic point entropy are given by, 
\[ \topent_p(\Omega(A_+))=\topent(\Omega(A_+))=\frac{1}{2}\cdot \intop_0^1 \intop_0^1 \log(4-2\cos(2\pi x) -2\cos (2\pi y))dxdy. \]
\end{reptheorem}

In Section 5 we focus on the topological entropy of SFTs of $\Z^d$-permutations  with movements restricted by some finite set $A$.

\begin{reptheorem}{th:AffInvEnt}
For any finite $A\subseteq \Z^2$ and injective affine transformation $T:\Z^d\to \Z^d$,
\[ \topent(\Omega(A))=\topent (\Omega(T(A))).\] 
\end{reptheorem}

In other words, the topological entropy is invariant under affine transformations of $\Z^d$. We use this invariance property in order to prove that the entropy is the same for all restricting sets which consists of $3$ points which form a triangle in $\Z^d$.

\begin{reptheorem}{th:AffEqCEnt}
Let $d\geq 2$ and $A,B\subseteq \Z^d$ be finite sets with full affine dimension such that $\abs{B}=\abs{A}=d'\leq d+1$.
Then, $\topent(\Omega(A))=\topent(\Omega(B))$. Furthermore, If $d=2$ and $d'=3$
\[ \topent(\Omega(A))=\topent(\Omega(B))=\frac{1}{4\pi^2}\intop_0^{2\pi}\intop_0^{2\pi}\log \abs{1+e^{ix}+e^{iy}}dxdy. \]
\end{reptheorem}

This gives rise to a universal lower bound for the entropy of permutations restricted by sets which are not contained in a line.

\begin{repcor}{cor:UnifBound} 
For $A\subseteq \Z^d$ which is not contained in a line 
\[  \topent(\Omega(A))\geq \frac{1}{4\pi^2}\intop_0^{2\pi}\intop_0^{2\pi}\log\abs{1+e^{ix}+e^{iy}}dxdy. \]
\end{repcor}

In the last part of the work, we discuss the related models of restricted injective and surjective functions on directed graphs. Given a directed graph $G=(V,E)$, we denote by $\Omega_I(G)$  the set of injective functions $V\to V$ restricted by the graph $G$ (in the same manner as we defined for permutations). Similarly, $\Omega_S(G)$ denotes the set of surjective functions restricted by the graph $G$. In a similar fashion to the case of permutations, $\Omega_I(G)$ and $\Omega_S(G)$ have the structure of topological dynamical systems. If $G$ is of the form $G=(\Z^d, E_A)$ for some finite $A\subseteq \Z^d$, $\Omega_I(G_A)$ and $\Omega_S(G_A)$ are SFTs and we use the notation of $\Omega_I(A)$ and $\Omega_S(A)$ for them. We prove that the invariant probability measures of $\Omega_S(A)$ and $\Omega_I(A)$ are supported on the set of $\Z^d$-permutations restricted by $A$. Formally,

\begin{repprop}{prop:SuppMeasure}
If $\pi$ is a random function on $\Z^d$ restricted by $A$ which is either almost surely injective or almost surely surjective and its distribution is shift-invariant, then almost-surely $\pi$ is a permutation of $\Z^d$. Equivalently: The support of any shift-invariant measure on $\Omega_I(A)$ or $\Omega_S(A)$ is contained in $\Omega(A)$.
\end{repprop}

Using the variational principle and Proposition~\ref{prop:SuppMeasure} we prove the following:
 
\begin{reptheorem}{th:InjSurj}
For any finite non-empty set  $A\subseteq\Z^d$, 
\[ \topent(\Omega(A))=\topent(\Omega_I(A))=\topent(\Omega_S(A)). \] 
\end{reptheorem}

\begin{flushleft}
	\textbf{Acknowledgments}
\end{flushleft}
The author thanks Tom Meyerovitch and Moshe Schwartz for their guidance during this work. Also,
the author thanks Ron Peled for a helpful discussion regarding the connection between restricted permutations and the double dimer model on the square lattice. 
This work was supported by the Israel Science Foundation (ISF) under grant No. 270/18 and 1052/18.


\section{Preliminaries}
\label{CHA:preliminaries}
A topological dynamical system is a pair $(X,\Gamma)$, where $X$ is a topological compact space (which is usually also a metric space), and $\Gamma$ is a semigroup, acting on $X$ by continuous transformations. In this work we investigate topological dynamical systems defined by permutations of graphs.  

\begin{definition}
Let $G=(V,E)$ be a countable locally finite directed graph. The set of permutations of $V$ restricted by $G$, denoted by $\Omega(G)$ is defined to be 
\[ \Omega(G)\eqdef \mathset{\pi\in S(V) : \forall v\in V, (v,\pi(v))\in E } ,\] 
where $S(V)$ denotes the set of permutations of $V$. A permutation in $\Omega(G)$ is said to be restricted by the graph $G$.
\end{definition}

In the context of restricted permutations, we view undirected graphs as symmetric directed graphs. That is, for an undirected graph $G=(V,E)$, the set of $G$-restricted permutations is defined to be $\Omega(G)\eqdef \Omega(\vec{G})$, where $\vec{G}=(\vec{V},\vec{E})$ is the directed graph defined by 
\[ \vec{V}=V, \  \vec{E}=\mathset{(v,u)\in V\times V : \mathset{v,u}\in E}.\]
We view $\Omega(G)$ as a subset of $V^V$, and thus as a topological space with the topology inherited from the product topology on $V^V$ (where $V$ has the discrete topology). Using the fact $G$ is locally finite, it is   easy to verify that $\Omega(G)$ is a compact space.

Let $\Is(G)$ denote the group of graph isomorphisms  of $G$.
Then a group $\Gamma\subseteq \Is(G)$ acts on $\Omega(G)$ by conjugation, namely $(\gamma,\pi) \mapsto \gamma(\pi)=\gamma \circ  \pi \circ  \gamma^{-1}$ defines a continuous left action of $\Gamma$ on $\Omega(G)$.

\begin{example}
\label{ex:Honeycomb}
Consider the undirected graph $L_H=(V_H,E_H)$ where
\[ V_H = \mathset{v + m\cdot (\sqrt{3},0)+n\cdot \parenv{\frac{\sqrt{3}}{2},\frac{3}{2}}:m,n\in \Z, v\in \mathset{\parenv{\frac{\sqrt{3}}{2},\frac{1}{2}},(\sqrt{3},1)}}\]   
and any vertex $v$ is connected to its three closest neighbours in $V_H$. That is, a vertex of the form $v=\parenv{\frac{\sqrt{3}}{2},\frac{1}{2}}+ m\cdot (\sqrt{3},0)+n\cdot \parenv{\frac{\sqrt{3}}{2},\frac{3}{2}}$ is connected in $E_H$ to $v+u$ where $u\in\mathset{\parenv{\pm \frac{\sqrt{3}}{2},\frac{1}{2}}, (0,-1) }$.

This graph is the well known two-dimensional honeycomb lattice (see Figure~\ref{fig:Honey_basic}). We have $\Z^2$ acting on $L_H$ by translations of the fundamental domain. By this we mean that $n=(n_1,n_2)\in \Z^2$ acts on a vertex $v\in V_H$ by 
 \[ n(v)\eqdef v + n_1\cdot (\sqrt{3},0)+n_2 \cdot \parenv{\frac{\sqrt{3}}{2},\frac{3}{2}}, \] 
see Figure~\ref{fig:Honey_fondemnetal}. We note that $L_H$ is a bipartite graph.

\begin{figure}
    \centering
    \subfigure[]
    {
        \includegraphics[scale=.13]{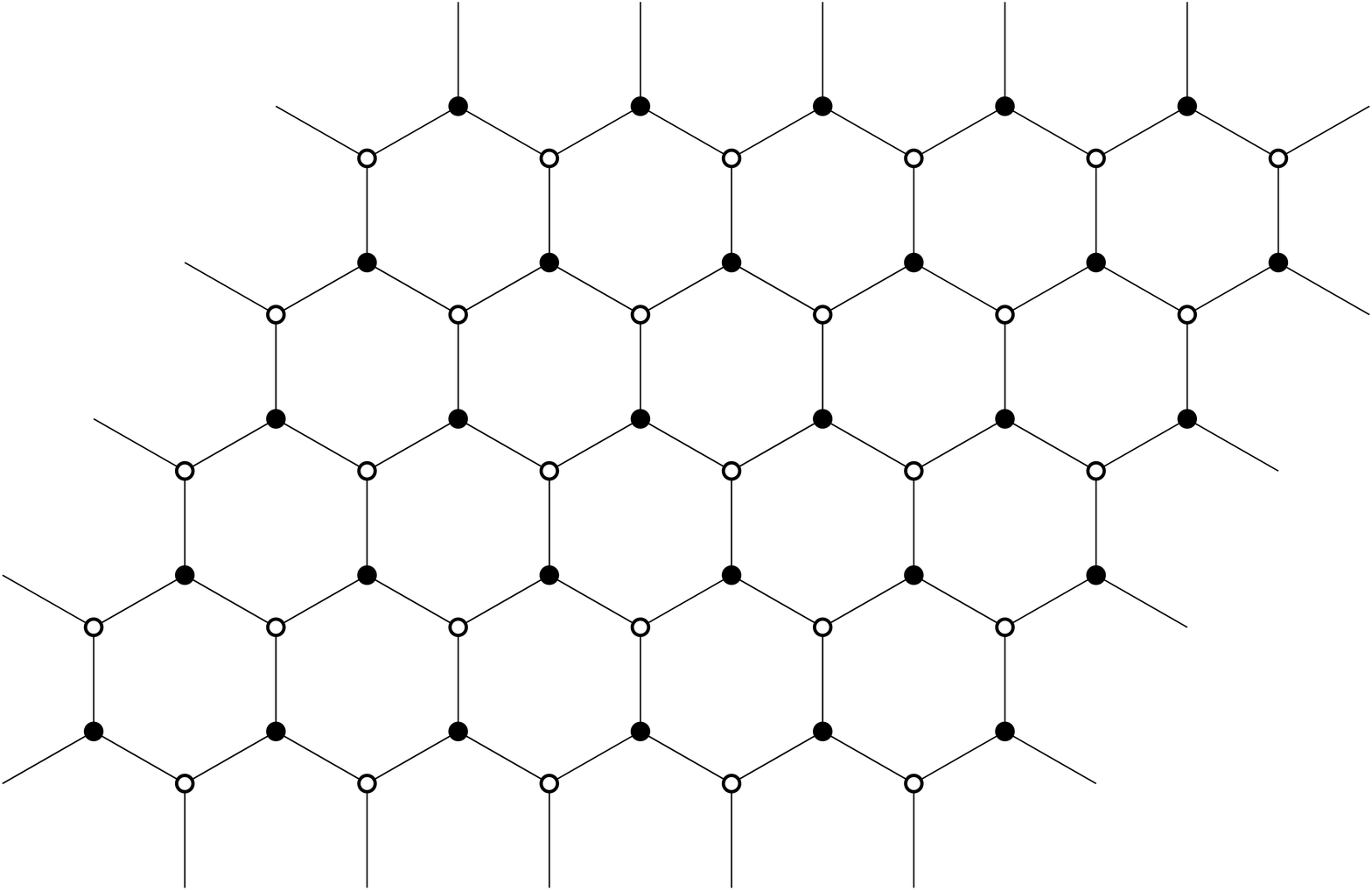}
        \label{fig:Honey_basic}
    }
    \subfigure[]
    {
        \includegraphics[scale=.37]{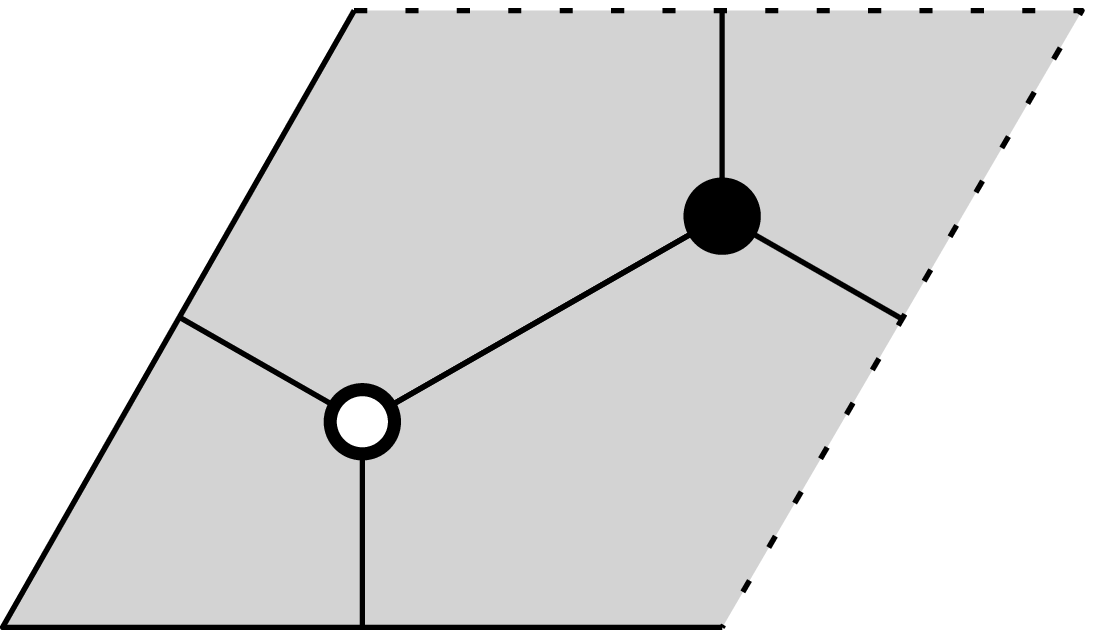}
        \label{fig:Honey_fondemnetal}
    }
    \caption
    {
        (a) The two-dimensional honeycomb lattice.
        (b) The fundamental domain.
    }
    \label{fig:Honey}
\end{figure}
\end{example}

\begin{example}
\label{ex:GenToZd1}
Let $\Gamma$ be a countable discrete group and $A\subseteq \Gamma$ be a finite non-empty set.  Consider the graph $G=(\Gamma,E)$, where 
\[ E \eqdef \mathset{(\gamma,\gamma a) : \gamma \in \Gamma, a\in A }.\] 
We have $\Gamma$ acting on $G$ by multiplication from the left. That is, $\alpha \in \Gamma$ acts on $\gamma\in \Gamma$ by $\alpha \cdot \gamma$. Note that for any $\gamma_1, \gamma_2\in A$ we have 
\begin{align*}
(\gamma_1,\gamma_2)\in E &\iff \exists a\in A \such \gamma_1=\gamma_2  a \\
&\iff  \exists a\in A \such \alpha \gamma_1=\alpha \gamma_2  a \\
& \iff (\alpha(\gamma_1),\alpha(\gamma_2))\in E.
\end{align*} 
This shows that $\Gamma$ acts by graph isomorphisms. Since $\Gamma$ is countable and $A$ is finite, $G$ is a locally finite graph. 
\end{example}

Consider the case when we choose the group from Example~\ref{ex:GenToZd1} to be $\Gamma=\Z^d$ for some $d\in \N$, and we take $A\subseteq \Z^d$ to be some finite set. In that case we have the graph  $G_A\eqdef(\Z^d,E_A)$, where 
\[ E_A\eqdef \mathset{(n,m)\in \Z^d\times \Z^d : m-n\in A},\] 
and $\Z^d$ acts on $(\Z^d,E_A)$ by translations. That is, $n$ acts on $m$ by $\sigma_{n}(m)\eqdef m+n$.
If a permutation of $\Z^d$ is restricted by $G_A$, we say that it is restricted by the set $A$.

The first case presented in Example~\ref{ex:GenToZd} was first visited by Schmidt and Strasser in \cite{SchStr17}, and will be a running example throughout this work. 
\begin{example}
\label{ex:GenToZd}
Let $d=2$ and consider the sets $A_+\eqdef \mathset{(\pm 1,0),(0,\pm 1)}\subseteq \Z^2,$ and $ A_L\eqdef \mathset{(0,0),(1,0),(0, 1)}\subseteq \Z^2$.

For a permutation $\pi \in \Omega(G_{A_L})$, we consider the orbit of an element $m\in \Z^d$, which is  $\parenv{\pi^{ n}(m)}_{n\in \Z}$. We note that the orbit of any point is either a single point or a bi-infinite sequence.

We can represent each infinite orbit of $\pi$ by a polygonal path in $\Z^2$, moving either north or east at each step. We can characterize $\pi$ by the configuration of non-intersecting polygonal paths in $\Z^2$ defined by its bi-infinite orbits. On the other hand, any set of such polygonal paths defines an element in $\Omega(G_{A_L})$ (see Figure~\ref{fig:Paths_L}).

In a similar fashion, we can represent a permutation in $\Omega(G_{A_+})$ by its orbits. In that case, orbits can be infinite, or finite with size greater than one. Each permutation in $\Omega(G_{A_+})$ corresponds to a covering of $\Z^2$ by transpositions (orbits of size 2) and  polygonal paths moving north, south, east or west at each step (see Figure~\ref{fig:Paths_+}).
In Figure~\ref{fig:Graph_PL} we exhibit the directed graph associated with $G_{A_+}$ and $G_{A_L}$.

\begin{figure}
    \centering
    \subfigure[]
    {
        \includegraphics[scale=.4]{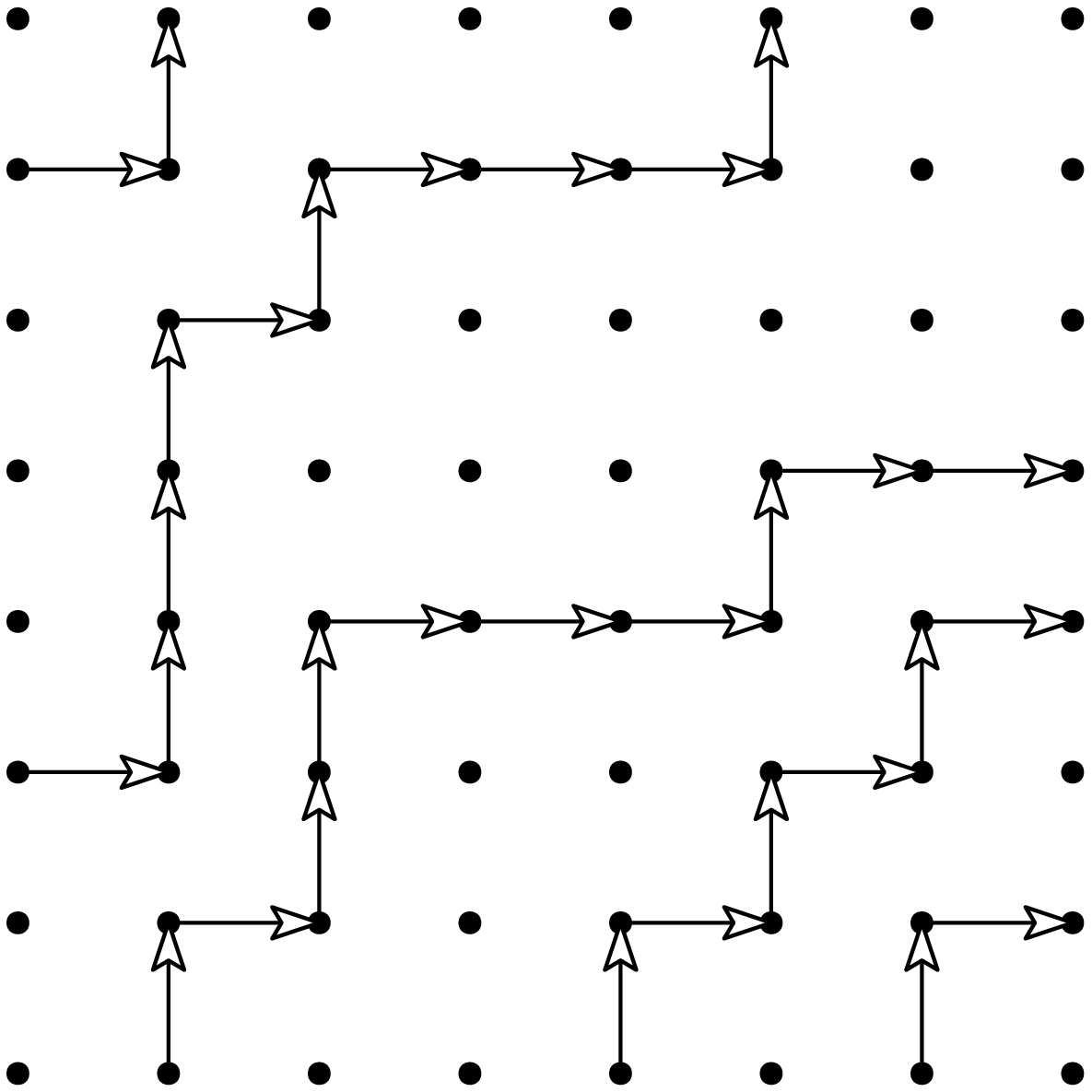}
        \label{fig:Paths_L}
    }  \hspace{7mm}
    \subfigure[]
    {
        \includegraphics[scale=.4]{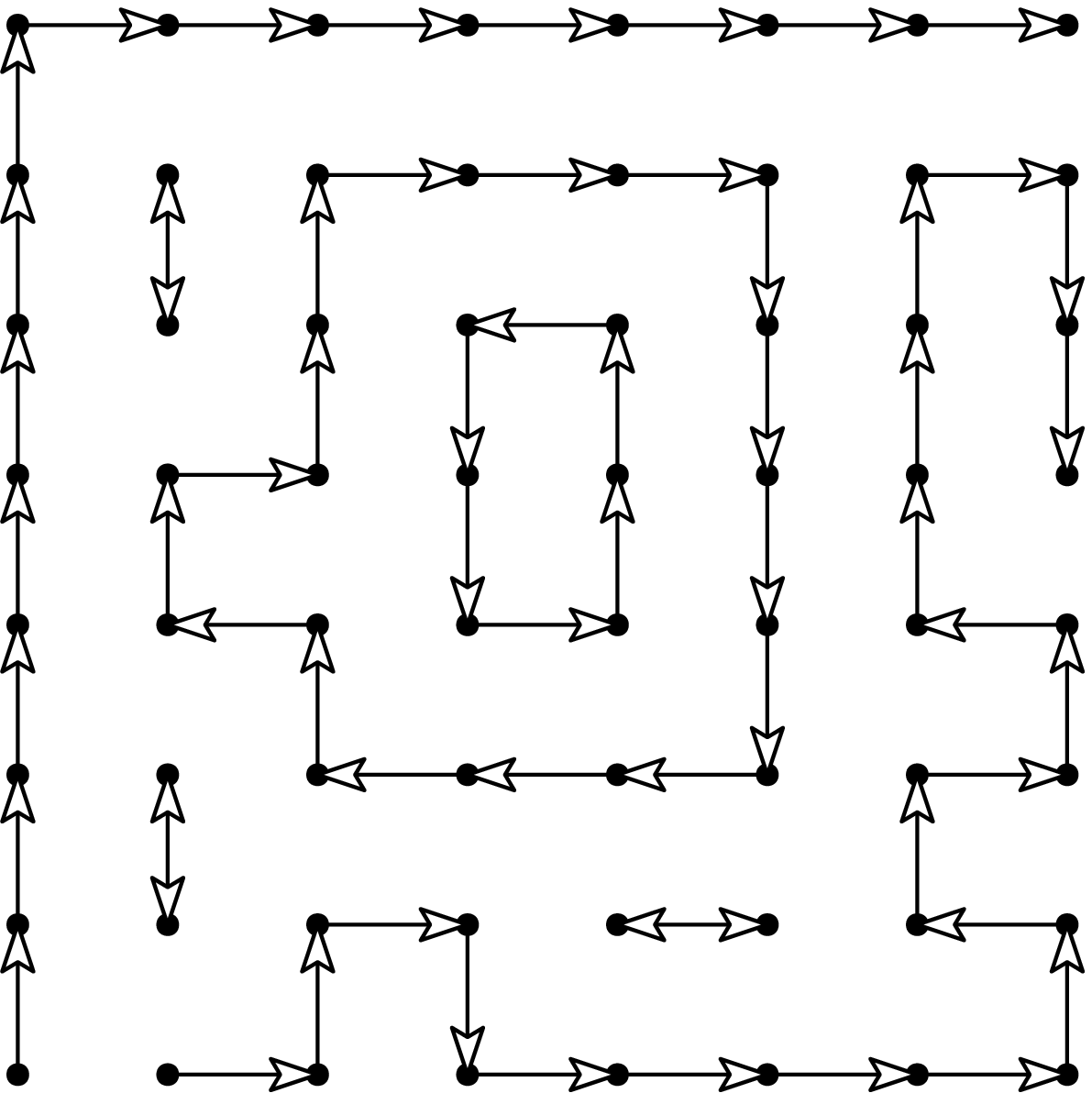}
        \label{fig:Paths_+}
    }
    \caption
    {
        (a) Paths configuration corresponding to an elements in $\Omega(G_{A_L})$.
        (b) Paths configuration corresponding to an elements in $\Omega(G_{A_+})$.
    }
    \label{fig:Paths}
\end{figure}

\begin{figure}
 \centering
  \includegraphics[width=150mm, scale=0.65]{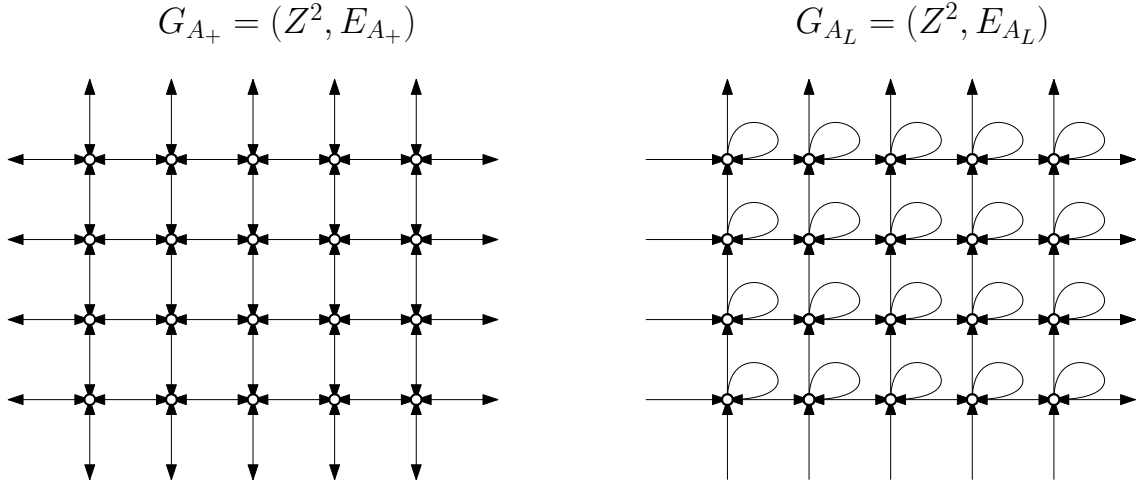}
  \caption{The graphs corresponding to $A_+$ and $A_L$.}
  \label{fig:Graph_PL}
\end{figure}
\end{example}

An important special case of dynamical systems is a subshift of finite type (SFT). Given a finite set, $\Sigma$, and an integer $d\in\mathbb{N}$, we consider the set $\Sigma^{\Z^d}$. An SFT, $\Omega \subseteq \Sigma^{\Z^d}$, is a subset of $\Sigma^{\Z^d}$, which is defined by a finite set of forbidden patterns. That is, there exists a finite set of forbidden patterns, $F\subseteq \bigcup_{B\in \fin(\Z^d)} \Sigma^B$, such that  
\[ \Omega=\mathset{\omega\in \Sigma^{\Z^d} : \forall n\in \Z^d \text{ and } B\in\fin(\Z^d), (\omega\circ \sigma_n)(B)\notin F },\]
where $\omega\left(B\right)$ is the restriction of $\omega$ to the
coordinates contained in the set $A$ and $\fin(\Z^d)$ denotes the set of all finite subsets of $\Z^d$. Following  convention, for $\omega \in \Sigma^{\Z^d}$ and $n\in \Z^d$, we will denote the composition $\omega \circ \sigma_n$ by $\sigma_n(\omega)$.

Throughout this work, for $m,n\in \N$ we use the notation $[n]$ for the set $\mathset{0,1,2,\dots,n-1}$ and $[m,n]$ for the set 
 $\mathset{m,m+1\dots,n-1}$. If $n=(n_1,\dots,n_d)\in \N^d$ we denote the box $[n_1]\times [n_2]\times \cdots \times [n_d]$ by $[n]$. 

 For a multi-index $n\in \N^d$, the set of $[n]$ patterns appearing in elements of $\Omega$ is denoted by $B_n(\Omega)$. Formally,
\[B_n(\Omega) \eqdef \mathset{P\in \Sigma^{[n]}  : \exists \omega\in \Omega \such \omega([n])=P}. \]

The topological entropy of an SFT $\Omega$ with the action of $\Z^d$ by shifts ($n$ acts by $\sigma_n$) is defined to be 
\[ \topent (\Omega) \eqdef \limsup_{n\to \infty} \frac{\log\abs{B_n(\Omega)}}{\abs{[n]}},\]
where we define that a sequence  $(n_k)_{k=1}^{\infty}\subseteq \N^d$ converge to $\infty$ if $n_k(i)\tends{k} \infty$ for all $1\leq i \leq d$.

Topological entropy is generally defined whenever $\Z^d$ acts continuously on a compact metrizable space, see \cite[Chapter 5]{Schm12} for the general definition and a detailed discussion. The  above definition of topological entropy  is equivalent in the particular case where $\Z^d$ acts on an SFT by shifts, in which we will focus throughout this work.

\begin{fact} 
\label{prop:fekete} \cite[ Section 2.2]{HocMey10} The $\limsup$ defining the topological entropy is actually a limit and \[ \topent(\Omega)=\inf_{n\in \N^d} \frac{\log \abs{B_n(\Omega)}}{\abs{[n]}}.\]
\end{fact}

Two SFTs, $\Omega_1 \subseteq \Sigma_1^{\Z^{d}}$ and $,\Omega_2 \subseteq \Sigma_2^{\Z^{d}}$, are said to be conjugate if there exists a homeomorphism $\Phi:\Omega_1\to \Omega_2$ that commutes with the action of $\Z^d$. Such a map is called a conjugacy. 

\begin{fact} 
\label{fact:conj} \cite[Chapter 1]{KerHan11} If $\Omega_1\subseteq \Sigma_1^{\Z^{d}}$ and $\Omega_2\subseteq \Sigma_2^{\Z^{d}}$ are conjugate, then $\topent(\Omega_1)=\topent(\Omega_2)$.
\end{fact}

The model of $\Z^d$ permutations restricted by some finite set, presented in Example~\ref{ex:GenToZd}, which will be the main focus of this work, was introduced by Schmidt and Strasser in \cite{SchStr17}. A permutation $\pi\in \Omega(A)$ for some finite $A\subseteq \Z^d$ can be identified with an element $\omega_\pi \in A^{\Z^d}$, where $\omega_\pi(n)=\pi(n)-n\in A$. This identification induces an embedding of $\Omega(G_A)$ in $A^{\Z^d}$, which we denote by $\Omega(A)$. Formally,
\[ \Omega(A)\eqdef \mathset{\omega_\pi : \pi\in \Omega(G_A)} =\{ \omega \in A^{\mathbb{Z}^d}~:~ n \mapsto (n+\omega(n)) \mbox{ is a permutation of } \mathbb{Z}^d\}.\]

From now on, we will use this notation in order to describe $\Z^d$-restricted permutations.
With this embedding, from a simple calculation it follows that the action of $\Z^d$ on $\Omega(G_A)$ by $m(\pi)=\sigma_m^{-1} \circ \pi \circ \sigma_m$ translates to a shift operation in $\Omega(A)$. That is, $\omega_{m(\pi)}(n)=(m(\omega_\pi))(n)$.

In their work \cite{SchStr17}, Schmidt and Strasser have shown that $\Omega(A)$ (with the shift operation) is an SFT for any finite $A\subseteq \Z^d$. They investigated the dynamical properties of such SFTs and their entropy, in general, and in some specific examples. We will focus on studying the entropy, mostly in the two-dimensional cases. 

\begin{definition}
Given a finite  set $A\subseteq \Z^d$ and $n=(n_1,n_2,\dots,n_d)\in \N^d$, a function $f:[n]\to \Z^d$ is said to be a permutation of the $n_1\times \cdots \times n_d$ discrete torus if $\tilde{f}:[n]\to [n]$ defined by
\[ \tilde{f}(m)=f(m) \bmod n,\]
is a permutation of $[n]$, where $m\bmod n=\left(m_{1}\bmod n_{1},m_{2}\bmod n_{2},\dots,m_{d}\bmod n_{d}\right)$ is the modulus of $m\in \Z^d$ from $n\in \N^d$. If $f$ is restricted by $A$,  we say that $f$ is a restricted permutation (by $A$) of the torus.
\end{definition}

\begin{definition}
Let $\Omega\subseteq \Sigma^{\Z^d}$ be a d-dimensional SFT over some finite alphabet $\Sigma$. For  a subgroup $\Gamma\subseteq \Z^d$ of finite index, we denote the set of $\Gamma$ periodic points by
\[ \fix_\Gamma(\Omega)\eqdef \mathset{\omega\in \Omega ~:~ \omega \circ \sigma_n=\omega \text{ for all }n\in \Gamma}.  \]
\end{definition}
Given a finite set $A\subseteq \Z^d$ and $n=(n_1,n_2,\dots,n_d)\in \N^d$, consider the group 
\[ \Gamma_n\eqdef n_1\Z\times n_2 \Z \times \cdots \times n_d\Z \subseteq \Z^d.\]

We observe that as long as $\abs{A\bmod n}=\abs{A}$, elements in $\fix_{\Gamma_n}(\Omega(A))$ correspond bijectively to restricted permutations of the $n_1\times\cdots\times n_d$ discrete torus, in the usual manner. 
We identify $\omega\in \fix_{\Gamma_n}(\Omega(A))$ with the function defined by the restriction of $\omega$ to $[n]$, denoted by $f_\omega$, which is, a restricted permutation of the torus.  That is, $\tilde{f}_\omega\eqdef f_\omega \bmod n$  is a permutation of $[n]$. 

\begin{definition}
The periodic entropy of an SFT $\Omega\subseteq \Sigma^{\Z^d}$ is defined to be 
\[ \topent_p(\Omega)\eqdef \limsup_{n\to \infty}\frac{\log \abs{\fix_{\Gamma_n}(\Omega)}}{\abs{[n]}}. \] 
\end{definition}

\begin{fact} \label{prop:GenPerEnt}
\cite[Proposition 4.1.15, Theorem 4.3.6]{LinMar85}  For an SFT $\Omega\subseteq \Sigma^{Z^d}$, 
\[ \topent_p(\Omega)\leq \topent(\Omega).\]
Furthermore, if $d=1$, equality holds.
\end{fact}

\begin{remark}
\label{re:1}
 The inequality from Fact~\ref{fact:conj} holds in the more general settings of shift spaces, in any dimension.  While equality holds in the one dimensional case, it can fail badly for general $\Z^d$ shift spaces when $d>1$, since there exists $\Z^d$ shift spaces with positive entropy and no periodic points (see \cite[ Section 9]{HocMey10}).
\end{remark}

Consider $n\in \N^d$ and a permutation $f\in S([n])$, that is, a permutation of the $n_1\times\cdots \times n_d$ array. We observe that $f$ is also a permutation of the $n_1\times\cdots \times n_d$ discrete torus, as $\tilde{f}=f \bmod n$  is a permutation of $[n]$, since $f=\tilde{f}$. Thus, for some finite $A\subseteq \Z^d$, denoting
\[ \Pi_n(A)\eqdef \mathset{\pi \in S([n])~:~\forall m\in [n], \pi(m)-m\in A },\] 
we have that $\Pi_n(A)$ is a subset of permutations of the $[n]$-discrete torus restricted by $A$ and therefore can be identified with a subset of $\fix_{\Gamma_n}$.
We conclude that $\abs{\Pi_n(A)} \leq \abs{\fix_{\Gamma_n}(\Omega(A))}. $
Given a finite  set $A\subseteq \Z^d$ we denote the exponential growth rate of $A$-restricted rectangular permutations by $\topent_c(A)$. That is, 
\[ \topent_c(A)\eqdef \limsup_{n\to\infty} \frac{\log\abs{\Pi_n(A)}}{\abs{[n]}}.\]
We call it the closed entropy of $A$-restricted permutations.
Followed by this observation and Fact~\ref{prop:GenPerEnt}, we have
 \[\topent_c(A) \leq \topent_p(\Omega(A))\leq \topent(\Omega(A)).\]

We now have three entropy-like quantities associated to permutations restricted by a fixed finite subset $A\subseteq \Z^d$: $\topent(\Omega(A)),\topent_p(\Omega(A))$ and $\topent_c(A)$. In the next sections, we will further study the relations between them.


\section{Restricted Permutations and Perfect Matchings}
\label{CHA:PM}
	Let  $G=(V,E)$ be an undirected graph. A perfect matching of $G$ is a subset of edges in which every vertex $v\in V$ is covered by exactly one edge. That is, $M\subseteq{E}$ is a perfect matching of $G$ if for every vertex $v\in V$ there exists a unique edge $e_v \in M$ for which $v\in e_v$. We denote the set of perfect matchings of a graph by $\PM(G)$. The set of perfect matchings of some locally finite graph $G=(V,E)$ can be naturally identified with a subset of $\prod_{v\in V}E(v)$ where $E(v)$ is the set of edges containing $v$. A perfect matching $M\in \PM(G)$ corresponds to a an element $w_M$ where $w_M(v)$ is the unique edge in $M$ which covers $v$. Thus, $\PM(G)$ can be considered as topological space with the topology inherited from the product topology on $\prod_{v\in V}E(v)$ (where $E(v)$ has the discrete topology for all $v\in V$). Since $G$ is locally finite $\prod_{v\in V}E(v)$ is compact and therefore $\PM(G)$ is compact as a closed subset $\prod_{v\in V}E(v)$.

	A weight function on the edges $W:E\to \C$ naturally defines a function on perfect matchings by
\[ W(M)\eqdef \prod_{e\in M }W(e).\] 
The perfect matching polynomial of $G$ with respect to $W$ is defined to be 
\[ \PM(G,W)\eqdef \sum_{M\in \PM(G)} W(M )=\sum_{M\in \PM(G)} \prod_{e\in M} W(e).\] 
We note that for the constant function $W \equiv 1$, $\PM(G,1)$ is just the number of perfect matchings of $G$.

 In \cite{Kas61, Kas63}, Kasteleyn presented an ingenious method for computing the perfect
matching polynomial of finite planar graphs. This method was used by Kasteleyn himself in order to compute the exponential growth rate of the number of perfect matchings of the two-dimensional square lattice. In 2006, Kenyon, Okounkov, and Sheffield \cite{KenOkoShe06} computed the exponential growth rate of perfect matchings of $\Z^2$-periodic bipartite planar graphs. In their work, they were also using Kasteleyn's method.

In this section we show two different characterizations of restricted permutations by perfect matchings (Theorem~\ref{th:GenCan} and  Theorem~\ref{th:PerToPM}). We use the results on perfect matchings of $\Z^2$-periodic bipartite planar graphs in order to compute the topological entropy of restricted permutations in a couple of two-dimensional cases (see Section~\ref{A_L} and Section~\ref{A_+}). We show an application Kasteleyn's method and present a polynomial-time algorithm for computing the exact number of $n \times n$-possible patterns in one specific case. Finally, we show a natural generalization of this algorithm (see Section~\ref{A_L}).

	\subsection{General Correspondence}
	\label{PM_char}
Let $G=(V,E)$ be a directed graph. Consider the undirected graph $G'=(V',E')$ defined by 
\[ V'\eqdef \mathset{I,O}\times V, \ \text{ and }\ E'\eqdef \mathset{\mathset{(O,v),(I,u) } ~:~ (v,u)\in E}. \]
Edges in $G'$ will be used to encode functions from $V$ to $V$ which are restricted by the original graph $G$. An edge of the form $\mathset{(v,O),(u,I)}$ will represent a mapping of $v$ to $u$. We note that $G'$ is always a bipartite graph, that will be important later on.  In Theorem~\ref{th:GenCan} we will show that perfect matchings of $G'$ correspond to restricted permutations of $G$.

Assume that a group $\Gamma$ is acting on $G$ by graph isomorphisms, one can define an action of $\Gamma$ on $G'$ by $\gamma (a,v)=(a,\gamma v)$, where $\ a\in \mathset{I,O}$ and 	$ \ v\in V$.
Clearly, this action is a group action on $G'$ and each element $\gamma\in \Gamma$ acts on $G'$ by graph isomorphism.

\begin{example}
Let $G=G_{A_L}$ be the graph described in Example~\ref{ex:GenToZd}. We recall that $\Omega(G_{A_L})$, also denoted by $\Omega(A_L)$, is the set of $\Z^2$-permutations restricted by the set $A_L\eqdef \mathset {(0,0),(0, 1), (1,0)}$. In that case, the graph $G'$ consists of two copies $\Z^2$ with edges between vertices whose difference are in $A_L$ (see Figure~\ref{fig:CanonL}).
\begin{figure}
 \centering
  \includegraphics[width=70mm, scale=0.5]{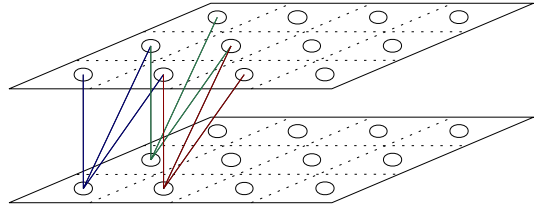}
  \caption{The graph $G_{A_L}'$}
  \label{fig:CanonL}
\end{figure} 
\end{example}

\begin{theorem}
	\label{th:GenCan}
There is a homeomorphism, $\Psi$, between the elements of $\Omega(G)$ and the perfect matchings in $\PM(G')$.
If  a group $\Gamma$ acts on $G$ by graph isomorphisms, then the action of $\Gamma$ on $G'$ induces a group action of $\Gamma$ on $\PM(G')$ such that $(\Omega(G),\Gamma)$ and $(\PM(G'),\Gamma)$ are topologically conjugated and $\Psi$ is a conjugation map. That is, $\Psi$ is an homeomorphism and the following diagram commutes
\[
\begin{tikzcd}
\Omega(G) \arrow[r,"\Gamma"] \arrow[d,swap,"\Psi"] &
 \Omega(G) \arrow[d,"\Psi"] \\
\PM(G')  \arrow[r,"\Gamma"] & \PM(G') 
\end{tikzcd}
\]
\end{theorem}

\begin{proof}
Consider the function $\Psi:\Omega(G)\to 2^{E'}$ defined by
\[ \Psi(\pi)\eqdef \mathset{\mathset{(O,v),(I,\pi(v))}~:~v\in V}.\]
Since $\pi$ is restricted by $G$, for any $v\in V$, $(v,\pi(v))\in E$. Thus, by the definition of $E'$, $\mathset{(O,v),(I,\pi(v))}\in E'$. This shows that $\Psi(\pi)\subseteq E'$.

We now show that  $\Psi(\pi)$ is a perfect matching of $G'$. Let $x$ be a vertex in $V'$. If $x$ is of the form $(O,v)$, $v\in V$, by the definition of $\Psi(\pi)$, $\mathset{(O,v),(I,\pi(v))}\in E'$ is the unique edge in $\Psi(\pi)$ containing $(O,v)$. If $x$ is of the form $(I,u)$, $u\in V$, we have that $\mathset{(O,\pi^{-1}(u)),(I,u)}\in E'$ is the unique edge which cover $(I,u)$, as $\pi$ is a bijection. 

For a perfect matching $M\in \PM(G')$ and $v\in V$ let $M(v)\in V$ be the unique vertex such that $\mathset{(O,v),(I,M(v))}\in M$. Consider the function $\Phi:\PM(G')\to \Omega(G)$ defined by \[ \Phi(M)(v)=M(v).\] It is easy to verify that $\Phi$ is well defined and that it is the inverse function of $\Psi$. Thus, $\Psi$ is a bijection. We note that $\Psi$ and $\Psi^{-1}$ are continuous as they are defined locally. Hence, $\Psi$ is a homeomorphism.

For the second part of the proof, let $\Gamma$ be a group, acting on $G$ by graph isomorphisms. 
For any $\gamma\in \Gamma$, the isomorphic action of $\gamma$ on $G'$ defines a map $\gamma:E'\to E'$ by 
\[\gamma(\mathset{(O,v),(I,u)})\eqdef \mathset{\gamma(O,v),\gamma(I,u)}\in E'.\]
This function maps perfect matchings of $G'$ to other perfect matchings of $G'$ as $\gamma$ acts on $G$ by graph isomorphism. It remains to show that the diagram commutes. Indeed, 
\begin{align*}
\Psi(\gamma(\pi))
&=\Psi(\gamma\circ \pi \circ \gamma^{-1} ))\\
&=\mathset{\mathset{(O,v),(I,\gamma(\pi(\gamma^{-1}v)))}~:~v\in V}\\
&=\mathset{\gamma\parenv{\mathset{(O,\gamma^{-1}v),(I,\pi(\gamma^{-1}v))}}~:~v\in V}\\
&=\gamma\parenv{\mathset{\mathset{(O,u),(I,\pi(u))}~:~u\in \gamma^{-1}V}}\\
\underset{\gamma^{-1}V=V}{\Rsh} &=\gamma \underset{\Psi(\pi)}{\underbrace{\parenv{\mathset{\mathset{(O,u),(I,\pi(u))}~:~u\in V}}}}=\gamma(\Psi(\pi)).
\end{align*}

\end{proof}

	\subsubsection{Permutations of $\Z^2$ Restricted by $A_L$ }
		\label{A_L}
 Permutations of $\Z^2$ restricted by the set $A_L=\mathset{(0,0),(0,1),(1,0)}$ were first studied by Schmidt and Strasser in \cite{SchStr17}. They proved that the topological entropy and the periodic entropy are equal in that case and speculated that it is around $\log(1.38)$ (which appears to be quite close to the exact value). We will show a connection between permutations of $\Z^2$ restricted by $A_L$ and perfect matchings of the honeycomb lattice and derive an exact expression for the topological entropy (and periodic entropy) of $\Omega(A_L)$. In the second part we describe a polynomial-time algorithm for computing the exact number of patterns in $B_{n,m}(A_L)$, and discuss a natural generalization of this algorithm.

By Theorem~\ref{th:GenCan}, we can  (bijectively) encode elements from $\Omega(A_L)$ by perfect matchings of the graph $G_{A_L}'$. If we draw the $G_{A_L}'$ on the plane, we may see that it is in fact the well known honeycomb lattice, $L_H$, which is a $\Z^2$-periodic bipartite planar graph (see Figure~\ref{fig:Honey_basic} (different colors of vertices represent the two disjoint and independent sets). By this, we mean that it can be embedded in the plane so that translations of the fundamental domain in $\Z^2$ act by color-preserving isomorphisms of $L_{H}$ -- isomorphisms which map black vertices to black vertices and white to white.

For $n\in \N$, let $L_{H,n}$ be the quotient of $L_{H}$ by the action of $n\Z^2$, which is a finite bipartite on the $n\times n$ torus (see Figure~\ref{fig:Honey_quotient}). A perfect matching of $L_{H,n}$ corresponds to a permutation of the $n\times n$ torus, restricted by $A$, in the same manner as in Theorem~\ref{th:GenCan}. Thus, 
\[ \abs{\fix_{n\Z^2}(\Omega(A_L))}=\abs{\PM(L_{H,n})}.\]

Kenyon, Okounkov and Sheffield \cite{KenOkoShe06} found an exact expression for the exponential growth rate of the number of toral perfect matchings of $\Z^2$-periodic bipartite planar graph. We use the following result which is a direct application of their work.
\begin{figure}
 \centering
  \includegraphics[width=70mm, scale=0.25]{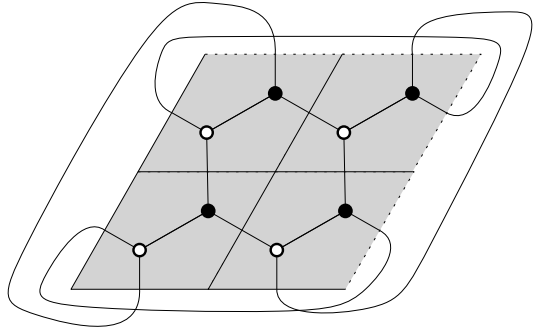}
  \caption{The quotient of $L_{H}$ by the action of $(2 \Z)^2$ }
  \label{fig:Honey_quotient}
\end{figure}
\begin{proposition} \cite{KenOkoShe06, Ken00}
\label{prop:HoneyEnt}
\[ \lim_{n\to \infty} \frac{\log\abs{\PM(L_{H,n})}}{n^2}=\frac{1}{4\pi^2}\intop_0^{2\pi}\intop_0^{2\pi}\log\abs{1+e^{ix}+e^{iy}}dxdy.\] 
\end{proposition} 
The connection between the periodic entropy and the topological entropy of $\Omega(A_L)$ was investigated by Schmidt and Strasser in their first work on restricted movement. They  proved the following proposition:
\begin{proposition} \cite{SchStr17}
\label{prop:PerFreeL}
\[ \limsup_{n\to \infty}\frac{\abs{\fix_{n\Z^2}(\Omega(A_L))}}{n^2}=\lim_{n\to \infty}\frac{\abs{\fix_{4n^3\Z^2}(\Omega(A_L))}}{(4n^3)^2}=\topent(\Omega(A_L)).\]
\end{proposition}
\begin{remark} 
The proof of Proposition~\ref{prop:PerFreeL} presented in \cite{SchStr17} by Schmidt and Strasser involves arguments regarding forming periodic points using reflections of polygonal patterns. Although using different machinery, the idea behind their proof is conceptually similar to the principle of reflection positivity, used by Meyerovitch and Chandgotia in \cite{MeyCha19} in order to explain that the topological entropy and the periodic entropy of the square lattice dimer model are equal. This suggests that the principle of reflection positivity may be used in order to prove that periodic entropy and topological entropy are equal in the more general case of perfect matchings of bipartite planar $\Z^2$-periodic graphs. 
\end{remark} 
We combine the results presented above with the observation about the correspondence between perfect matchings of the honeycomb lattice to $A_L$-restricted permutations to obtain:
\begin{theorem}
\label{th:SolL}
\[ \lim_{n\to\infty}\frac{\log\abs{\fix_{n\Z^2}(\Omega(A_L))}}{n^2}  =\topent_p(\Omega(A_L))=\topent(\Omega(A_L))=\frac{1}{4\pi^2}\intop_0^{2\pi}\intop_0^{2\pi}\log\abs{1+e^{ix}+e^{iy}}dxdy.\] 
\end{theorem}
\begin{proof}
From Fact~\ref{prop:GenPerEnt}, we know that $\topent_p(\Omega(A))\leq \topent(\Omega(A))$. 
On the other hand, by Proposition~\ref{prop:PerFreeL},
\begin{align*}
 \topent_p(\Omega(A))&=\limsup_{n_1,n_2\to\infty}\frac{\log\abs{\fix_{\Gamma_{(n_1,n_2)}}(\Omega(A))}}{n_1 n_2}\\
 &\geq \lim_{n\to\infty} \frac{\log\abs{\fix_{4n^3\Z^2}(\Omega(A))}}{(4n^3)^2}=\topent(\Omega(A)).
\end{align*}
This shows that 
\[\topent(\Omega(A))=\topent_p(\Omega(A))=\lim_{n\to \infty}\frac{\abs{\fix_{4n^3\Z^2}(\Omega(A))}}{(4n^3)^2}. \] 
Using Proposition~\ref{prop:HoneyEnt} and the equivalence between perfect matchings of $L_{H,n}$ and periodic restricted permutations, we conclude
\begin{align*}
\frac{1}{4\pi^2}\intop_0^{2\pi}\intop_0^{2\pi}\log\abs{1+e^{ix}+e^{iy}}dxdy&=\lim_{n\to \infty}\frac{\log\abs{\PM(L_{H,n})}}{n^2}\\
&=\lim_{n\to \infty}\frac{\log\abs{\fix_{n\Z^2}(\Omega(A))}}{n^2}\\
&=\lim_{n\to \infty}\frac{\log \abs{\fix_{4n^3\Z^2}(\Omega(A))}}{(4n^3)^2}=\topent_p(\Omega(A)).
\end{align*}
This completes the proof of the theorem.
\end{proof}
\begin{remark}
Theorem~\ref{th:SolL} provides a complete solution to the question raised in the work by Schmidt and Strasser \cite{SchStr17}, whether it is true that $ \lim_{n\to\infty} \frac{\log\abs{\fix_{n\Z^2}(\Omega(A))}}{n^2}$ exists (and is equal to $\topent(\Omega(A))$).
\end{remark}

\begin{remark}
Theorem \ref{th:SolL} verifies the numerical based guess for the value of $\topent(\Omega(A_L))$ suggested by Schmidt and Strasser \cite{SchStr17}. According to their numerical calculation, $\topent(\Omega(A_L))\approx \log(1.38)=0.322083\dots$, which is quite close to the exact value $\topent(\Omega(A_L))=0.323066\dots $, given by the double integral in Theorem \ref{th:SolL}.
\end{remark}

We saw that there exists a natural correspondence between permutations restricted by $A_L$ and perfect matchings of the honeycomb lattice. Unfortunately, this correspondence does not translate to a matching between elements in $B_{n,m}(A_L)$  and perfect matching of finite sub-graphs of the honeycomb lattice. Therefore, we cannot use Kasteleyn's method for counting perfect matchings in order to compute $\abs{B_{n,m}(A_L)}$. However, we show that patterns from $B_{n,m}(A_L)$ correspond to objects which we call perfect covers, that may be counted in polynomial time.

\begin{definition}
Let $G=(V,E)$ be an undirected graph and $\hat{V}\subseteq V$ be a subset of vertices. A set of edges $C\subseteq E$ is said to be a perfect cover of $\hat{V}$ if the following are satisfied:
\begin{itemize}
\item Any vertex $v\in \hat{V}$ has an edge $e_v\in C$ such that $v\in e_v$.
\item No two different edges in $C$ share a vertex.
\item For any edge $e \in C$, the intersection $e\cap \hat{V}$ is non-empty.  
\end{itemize} 
Denote the set of all perfect covers of $\hat{V}$ in $G$ by $\PC(\hat{V},G)$. 
\end{definition}

In Section~\ref{LocGlob} we show that locally admissible rectangular patterns in $\Omega(A_L)$ are also globally admissible. That is, the elements in $B_{n,m}(A_L)$ are injective functions $[n]\times [m] \to [n]\times [m]+A_L$ such that the inner square $[1,n]\times[1,m]$ is contained in their image.

Let $\pi\in \Omega(A_L)$. In the corresponding perfect matching of  $G_{A_L}'$, the edges connected to vertices in the set $\mathset{(O,k) ~:~ k\in [n]\times [m]}$ determine the restriction of $\pi$ to the rectangle $[n]\times [m]$. Thus, similarly to the proof of Theorem~\ref{th:GenCan}, perfect covers of the set \[ \hat{V}_{n,m}\eqdef \parenv{ \mathset{O}\times([n]\times [m])} \bigcup \parenv{\mathset{I}\times([1,n] \times [1,m])}\]
has a bijective correspondence with injective functions  $[n]\times [m] \to [n]\times [m]+A_L$ such that the inner square $[1,n]\times[1,m]$ is contained in their image (see Figure~\ref{fig:PerToPCnew}). Therefore, we have
 \[ \abs{\PC(\hat{V}_{n,m},G_{A_L})}=\abs{B_{n,m}(A_L)}.\]

\begin{figure}
 \centering
  \includegraphics[width=130mm, scale=0.5]{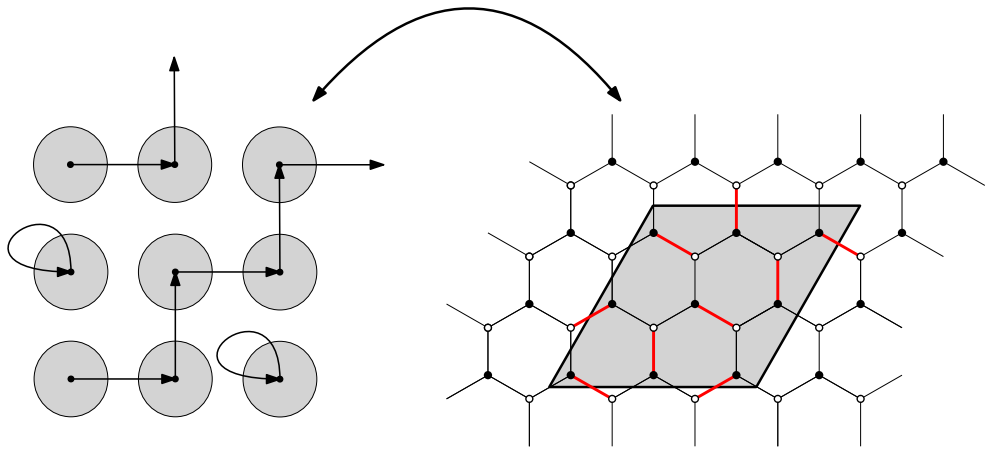}
  \caption{A correspondence between a function in $B_{3,3}(A_L)$ and a perfect cover of $\hat{V}_{3,3}$ in $G_{A_L}$. }
  \label{fig:PerToPCnew}
\end{figure}

\begin{proposition} \cite{Eli19}
\label{prop:GenAlg}
Let $G=(V,E)$ be a locally finite planar graph and $\hat{V}\subseteq V$ be a finite subset of vertices with even size that can be separated from $V\setminus \hat{V}$ by a simply connected domain (in some planar representation of $G$). Then, there exists a polynomial time algorithm for computing $\abs{\PC(\hat{V},G)}$.
\end{proposition}

The idea in the proof of Proposition~\ref{prop:GenAlg}, presented in detail in \cite{Eli19}, is that we can construct a finite undirected  planar graph $\hat{G}$, and a weight function $W$ on the edges in $\hat{G}$,  such that $\PM(\hat{G},W)=\abs{\PC(\hat{V},G)}$. 

Let $S$ be the set of all vertices in $V\setminus \hat{V}$ connected to some vertex in $\hat{V}$. Since $G$ is locally finite and $\hat{V}$ is a finite set, $S$ is finite as well. Let $s_1,\dots,s_n$ be a clockwise order enumeration of $S$ (with respect to the planar representation of $G$ in which $\hat{V}$ and $\hat{V}\setminus V$ are separated by a simply connected domain). In the new graph, $\hat{G}$, vertices of $\hat{V}$ and edges between them remain as in the original graph. We replace any vertex $s_i\in S$ by a gadget, the graph $K_4$ (see Figure~\ref{fig:L-GadgetsNew})). We denote this gadget by $T_i$, and add edges connecting $T_{i,1}$ with all the vertices from $\hat{V}$ connected to $s_i$ in the original graph $G$. Finally, we add the edges connecting between $T_{i,2}$ and $T_{i+1,3}$ for all $i=1,2,\dots,n-1$ (see Figure~\ref{fig:GenralAlg}). We set $W(e)=1$ for all of the edges in $\hat{G}$ beside the most inner edges of the $T$-gadgets, for which we set $W(e)=\frac{1}{3}$ (see Figure~\ref{fig:L-GadgetsNew}). In this new graph, $\PM(\hat{G},W)$ is equal to the number of perfect covers in $\PC(\hat{V},G)$ that contain an even number of edges that intersect $V\setminus \hat{V}$. Since we choose $\hat{V}$ such that $\abs{\hat{V}}$ is even, all of the perfect covers in $\PC(\hat{V},G)$ are such, and therefore $\PM(\hat{G},W)=|\PC(\hat{V},G)|$.

\begin{figure}
    \centering
    \subfigure[]
    {
  \includegraphics[width=25mm, scale=0.5]{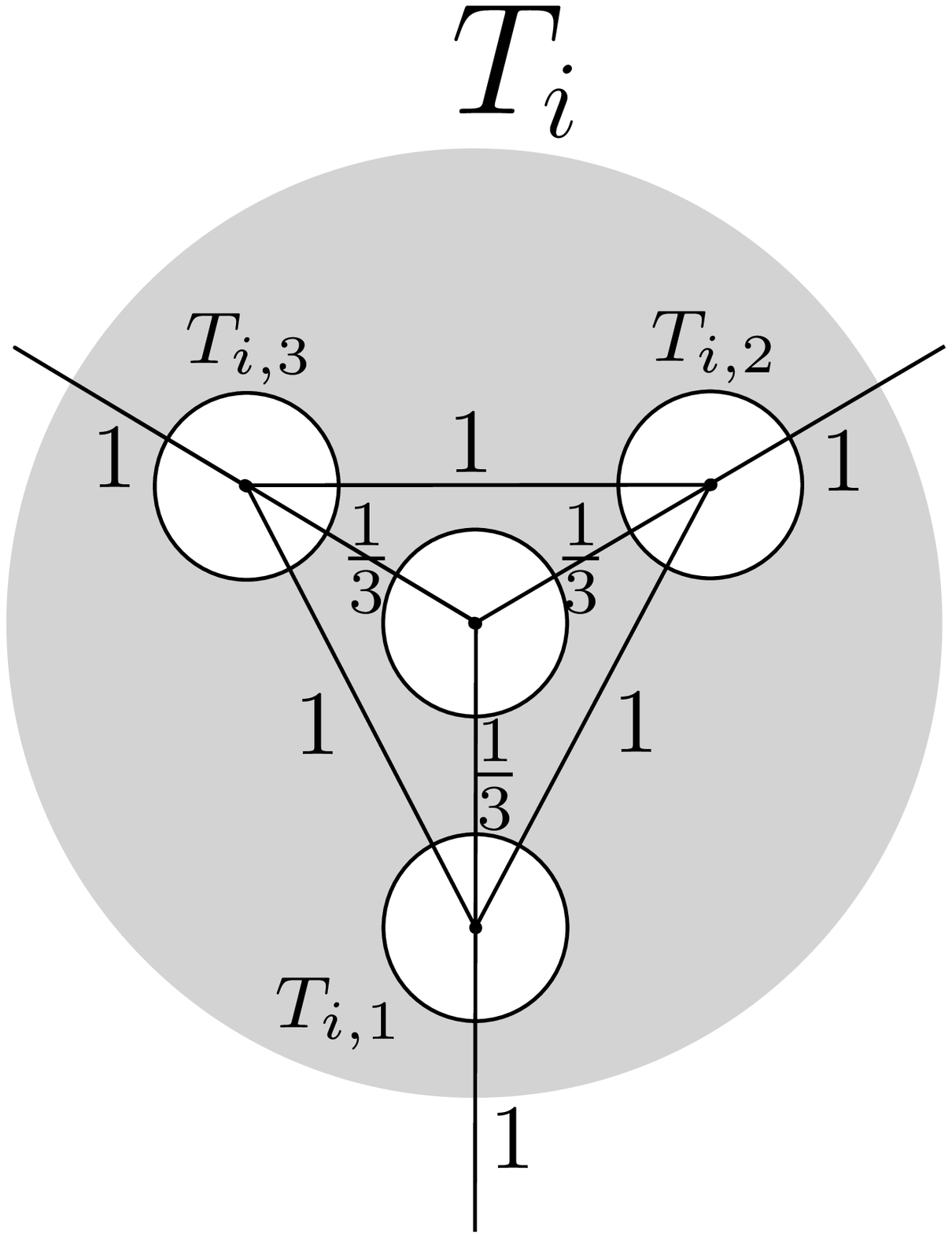}
  \label{fig:L-GadgetsNew}
    }
    \subfigure[]
    {
  \includegraphics[width=120mm, scale=0.65]{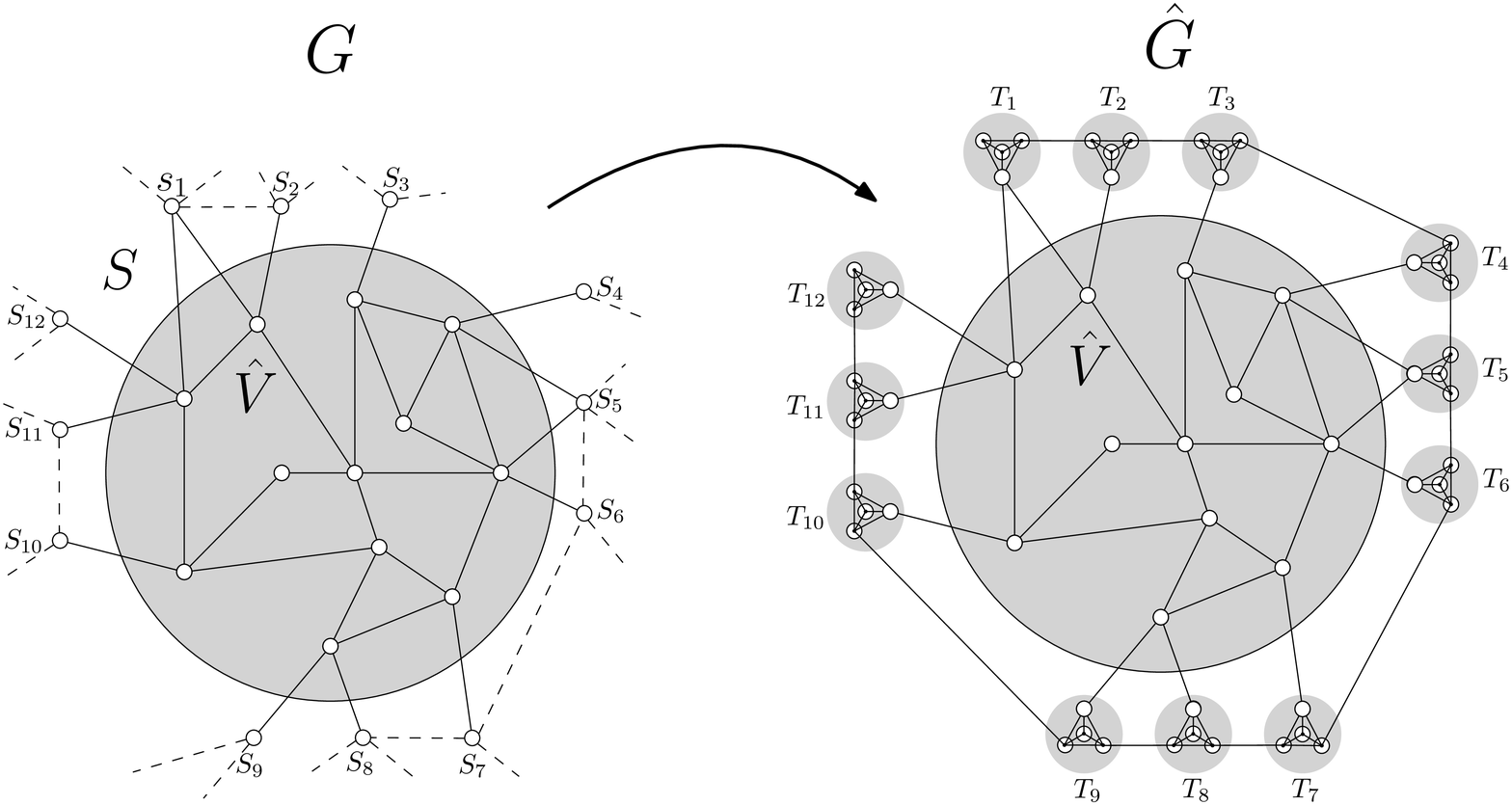}
  \label{fig:GenralAlg}
    }
    \caption
    {
        (a) The $T$-gadget and the weights on its edges.
        (b) The construction of $\hat{G}$ from $G$.
    }
    \label{fig:AlgG}
\end{figure}

In his work \cite{Kas63}, Kasteleyn showed that for any finite planar graph $G=(V,E)$, and a weight function on the edges $W$, it is possible to find an orientation of the edges such that $\abs{\PM(\hat{G},W)}=\abs{\det(A)}$, where $A$ is the $|V|\times |V|$ adjacency matrix of the orientation, defined by
\[ A (i,j)=\begin{cases}
W\parenv{\mathset{i,j}} & \text{If } i\to j \\
-W\parenv{\mathset{i,j}} & \text{If } j\to i\\
0 & \text{Otherwise}.
\end{cases}\] 
Such an orientation is called a Pfaffian orientation. The inductive proof of his theorem gives rise to an algorithm for finding a Pfaffian orientation in a complexity of $O(|E|^2)$ operations, and since in a planar graph $|E|=O(|V|)$, a Pfaffian orientation can be found in $O(|V|^2)$ operations.

We note that the number of vertices in our newly constructed graph $\hat{G}$ is bounded by $=C|\hat{V}|$ for some $C>0$, as the original graph is planar.  Therefore, applying Kasteleyn's method we can fined a Pfaffian orientation in $O(|\hat{V}|^2)$ operations and compute $\PM(\hat{G},W)$, which is the determinant of the adjacency matrix of this orientation. The determinant of a $|\hat{V}|\times |\hat{V}|$ matrix is computable by $O(|\hat{V}|^3)$ operations and therefore so as $\PM(\hat{G},W)$.

By Proposition~\ref{prop:GenAlg}, we can compute $\abs{\PC(\hat{V}_{n,m},G_{A_L})}$ in polynomial-time as $G_{A_L}$ is planar and we may separate between $ \hat{V}_{n,m}$ and ${V_{A_L}}\setminus  \hat{V}_{n,m}$ by a parallelogram (see Figure~\ref{fig:PerToPCnew}). This yields the following
\begin{proposition}
\label{prob:algA_l}
There exists an algorithm for computing $\abs{B_{n,m}(A_+)}$ with time complexity $O(m^3n^3)$.
\end{proposition}

We comment that these results bear some resemblance to the results in \cite{SchBru08} that use holographic reductions.

	\subsection{Alternative Correspondence For Bipartite Graphs}
	\label{PM_alt}
In the first part of this section we described a general correspondence of restricted permutations by perfect matching. This correspondence proved to be useful for studying cases where the corresponding graph is a $\Z^2$-periodic bipartite planar graph. Unfortunately, this is usually not the case. In this section, we find an alternative correspondence of restricted movement permutations and perfect matchings, for the case where the original graph is bipartite. We use this correspondence in order to study  restricted permutations of the graph $G_{A_+}$ presented in Example~\ref{ex:GenToZd}.

\begin{definition}
Let $G=(V,E)$ be a directed graph. The undirected version of $G$ is the undirected graph $\tilde{G}=(V,\tilde{E})$ obtained by removing the directions of edges in $V$. That is, $\tilde{E}\eqdef\mathset{\mathset{v,u}~:~ (v,u)\in E}$.
\end{definition}
\begin{theorem}
\label{th:PerToPM}
Let $G=(V\uplus U,E)$ be a directed bipartite graph. There is a continuous embedding of $\Omega(G)$ inside $\PM(\tilde{G})\times \PM(\tilde{G})$, where $\tilde{G}=(V\uplus U,\tilde{E})$ is the undirected version of $G$.
If a group $\Gamma$ acts on $G$ by bipartite graph isomorphisms (that is, $\gamma V=V$ for all $\gamma\in \Gamma$) then it induces an action on $\PM(\tilde{G})$ and the following diagram commutes: 
\[
\begin{tikzcd}
\Omega(G) \arrow[r,"\Gamma"] \arrow[d,swap,"\Phi"] &
 \Omega(G) \arrow[d,"\Phi"] \\
\PM(\tilde{G})^2  \arrow[r,"\Gamma"] & \PM(\tilde{G})^2 
\end{tikzcd}
\]
Furthermore, if $G$ is symmetric (that is, $(v,u)\in E$ implies $ (u,v)\in E$), $\Phi$ is a homeomorphism and $(\Omega(G),\Gamma)$ and $(\PM(\tilde{G}\times \tilde{G}),\Gamma) $ are topologically conjugated. 
\end{theorem}

\begin{proof}
Let $\pi\in \Omega(G)$ be a restricted permutation. Consider $M_{\pi}^1\eqdef \mathset{\mathset{v,\pi(v)}~:~v\in V} $
and $M_{\pi}^2\eqdef \mathset{\mathset{u,\pi(u)}~:~u\in U}$. 
Clearly, $M_{\pi}^1,M_{\pi}^2\subseteq \tilde{E}$ as $\pi$ is restricted by $G$. From the definition of $M_{\pi}^1$, since $\pi$ is bijective, for $x\in V\uplus U$, the unique edge which covers $x$ in $M_{\pi}^1$ is $\mathset{x,\pi(x)}$ if $x\in V$ and $\mathset{x,\pi^{-1}(x)}$ if $x\in U$. This shows that $M_{\pi}^1\in \PM (\tilde{G})$. Similarly we have $M_{\pi}^2\in \PM (\tilde{G})$.

Define $\Phi(\pi)=(M_{\pi}^1,M_{\pi}^2)$. We now turn to prove that $\Phi$ is injective. Let $\pi_1,\pi_2\in \Omega(G)$ be two distinct restricted permutations. Since $\pi_1\neq \pi_2$, there exists $x\in V\uplus U$ such that $\pi_1(x)\neq \pi_2(x)$, where without loss of generality $x\in V$. From the definition of $M_{\pi_1}^1$ and  $M_{\pi_2}^1$, we have that 
$\mathset{x,\pi_1(x)}\in M_{\pi_1}^1$ and   $\mathset{x,\pi_2(x)}\in M_{\pi_2}^1$. Since $M_{\pi_1}^1$ and $M_{\pi_2}^1$ are perfect matchings of $\tilde{G}$, $\mathset{x,\pi_2(x)}\notin M_{\pi_1}^1 $
as $x$ is already covered by the edge $\mathset{x,\pi_1(x)}$ in $M_{\pi_1}^1$. Thus, $ M_{\pi_1}^1\neq  M_{\pi_2}^1$ and in particular $\Phi(\pi_1)\neq \Phi(\pi_2)$. Similarly to the proof of Theorem \ref{th:GenCan}, $\Phi$ is continuous as it is defined locally.

Let $\Gamma$ be a group acting on $G$ by bipartite graph isomorphisms. For any $v\in V$, $u\in U$, and $\gamma\in \Gamma$ we have $\mathset{u,v}\in \tilde{E} \iff \mathset{\gamma u,\gamma v}\in \tilde{E}$. This implies that $\Gamma$ acts on $\tilde{G}$ by isomorphisms and similarly to the proof of Theorem~\ref{th:GenCan}, this action induces a group action of $G$ on $\PM(\tilde{G})$ by $ \gamma(M)\eqdef \mathset{\mathset{\gamma v,\gamma u}~:~\mathset{v,u}\in M}$.

In order to complete the first part of the theorem, it remains to show that for any $\pi\in \Omega(G)$ and $\gamma\in \Gamma$ we have $\Phi(\gamma(\pi))=\gamma(\Phi(\pi))$. That is, $(\gamma M_\pi^1, \gamma M_\pi^2)=(M_{\gamma(\pi)}^1,M_{\gamma(\pi)}^2)$. Indeed,
\begin{align*}
\gamma M_\pi^1&=\mathset{\mathset{\gamma v,\gamma \pi(v)}~:~v\in V}=\mathset{\mathset{\gamma v,(\gamma\pi \gamma^{-1})(\gamma v))}~:~v\in V}\\
&=\mathset{\mathset{u,(\gamma(\pi))(u))}~:~u\in \gamma V}\underset{(\gamma V=V)}{=}\mathset{\mathset{u,(\gamma(\pi))(u))}~:~u\in V}=M_{\gamma(\pi)}^1.
\end{align*}
Symmetrically, we show that $\gamma M_\pi^2=M_{\gamma(\pi)}^2$, which completes the first part of the proof.

Assume furthermore that $G$ is symmetric. We need to show that the map $\pi \to (M_\pi^1,M_\pi^2)$ is invertible. Given two perfect matchings, $M_1,M_2\in \PM (\tilde{G})$, for any $x\in V\uplus U$ and $i \in \mathset{1,2}$, denote by $M_i(x)$ the unique vertex in $V\uplus U$ such that $\mathset{x,M_i(x)}\in M_i'$. Define $\pi_{M_1,M_2}:V\uplus U \to V\uplus U$ by 
\[ 
\pi_{M_1,M_2}(x)=\begin{cases}
M_1(x) & \text {if } x\in V\\
M_2(x) & \text {if } x\in U.
\end{cases}
\] 
Note that for any $x\in V\uplus U$,
\[ \mathset{x,\pi_{M_1,M_2}(x)} \in \mathset{\mathset{x,M_1(x)},\mathset{x,M_2(x)}} \subseteq \tilde{E}.\]
Since $G$ is symmetric, $(x,M_1(x)),(x,M_2(x))\in E$ and therefore $(x,\pi_{M_1,M_2}(x))\in E$. That is, $\pi_{M_1,M_2}$ is restricted by $G$. Assume that $\pi_{M_1,M_2}(x)=\pi_{M_1,M_2}(x')$. Since $G$ is bipartite and $\pi_{M_1,M_2}$ is restricted by $G$ we have that $x,x'\in V$ or $x,x'\in U$, where without the loss of generality let us assume that $x,x'\in V$. From the definition of $\pi_{M_1,M_2}$, it follows that 
\[ M_1(x)=\pi_{M_1,M_2}(x)=\pi_{M_1,M_2}(x')=M_1(x'),\]
and therefore $x=x'$ (as $M_1$ is a perfect matching of $G$). This shows that $\pi_{M_1,M_2}$ is injective. Let $x\in U$, we note that $M_1(x)\in V$ and $M_1(M_1(x))=x$. From the definition of $\pi_{M_1,M_2}$, it follows that 
\[ \pi_{M_1,M_2}(M_1(x))=x.\]
Similarly, if $x\in V$, $\pi_{M_1,M_2}(M_2(x))=x$ and $\pi_{M_1,M_2}$ is onto $V\uplus U$. Define $\Psi:\PM(\tilde{G})^2\to \Omega(G)$ by $\Psi(M_1,M_2)=\pi_{M_1,M_2}$.
It is easy to see that $\Psi\circ \Phi$ and $\Phi\circ \Psi$ are the identity functions on $\PM(\tilde{G})^2$ and $\Omega(G)$ respectively. By similar argument as in the first part of the proof,  $\Psi$ is continuous. This shoes that $\Phi$ is a homeomorphism and completes the proof. 
\end{proof}

\begin{remark}
\label{re:bipartiso}
In the setting of Theorem~\ref{th:PerToPM}, if $\Gamma$ acts by graph isomorphisms which might switch between the sides of the graph, the theorem remains true when we set the action of $\Gamma$ on $\PM(\tilde{G})^2$ to be 
\[ \gamma (M_1,M_2)\eqdef\begin{cases}
(\gamma M_1,\gamma M_2) & \text{if } \gamma (V)=V\\
(\gamma M_2,\gamma M_1) & \text{if } \gamma (V)=U\\
\end{cases} \]
instead of $\gamma (M_1,M_2)=(\gamma M_1,\gamma M_2)$. The original proof will work with this minor change. Thus, we may generalize the definition of bipartite group actions on a graph $G$ and allow element to switch between sides. Defining the induced group action on $\PM(\tilde{G})^2$ as above, the statement in Theorem~\ref{th:PerToPM} holds.
\end{remark}

We recall that for an undirected graph $G=(V,E)$, the set of $G$-restricted permutations is defined to be $\Omega(G)\eqdef \Omega(\vec{G})$, where $\vec{G}=(\vec{V},\vec{E})$ is the directed version of $G$. That is, the graph defined by 
\[ \vec{V}=V, \  \vec{E}=\mathset{(v,u)\in V\times V : \mathset{v,u}\in E}.\]
We note that the undirected version of $\vec{G}$ is $G$. Therefore, by Theorem \ref{th:PerToPM}, we conclude the following:

\begin{corollary}
\label{cor:undirected-PM}
For any $G$ undirected  bipartite graph, there is a homeomorphism between $\Omega(G)$ and $\PM(G)^2$. If $\Gamma$ is acting on $G$ by bipartite graph isomorphisms, then $(\Omega(G),\Gamma)$ and $(\PM(G)^2,\Gamma)$ are topologically conjugated.  
\end{corollary}

\begin{example}
Consider the two-dimensional (undirected) honeycomb lattice from Example~\ref{ex:Honeycomb}, denoted by $L_H$. By Corollary~\ref{cor:undirected-PM}, restricted permutations of the honeycomb lattice correspond with pairs of perfect matchings of $L_H$.  In Section~\ref{A_L} we have shown that perfect matchings of the honeycomb lattice are in 1-1 correspondence with permutations of $\Z^2$ restricted by the set $A_L$. Combining the results, we conclude that restricted permutations of the honeycomb lattice correspond with pairs of $\Z^2$ permutations restricted by $A_L$. That is, $\Omega(G_{L_H})$ and $\Omega(A_L)\times \Omega(A_L)$ are topologically conjugated.
\begin{figure}
 \centering
  \includegraphics[width=120mm, scale=0.65]{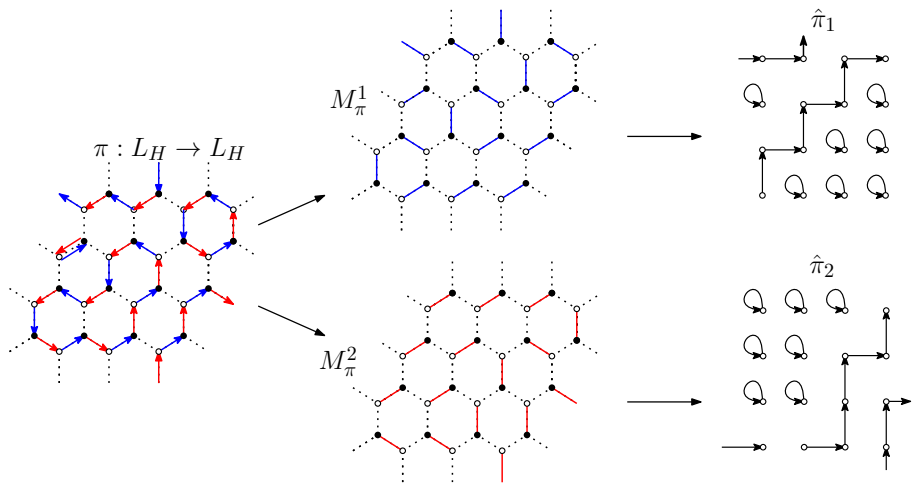}
  \caption{The correspondence between a restricted permutation of the honeycomb lattice, perfect matchings and permutations of $\Z^2$ restricted by $A_L$.}
  \label{fig:HoneycombPerToPMtoAL}
\end{figure}
\end{example}
		
	\subsubsection{Permutations of $\Z^2$ Restricted by $A_+$ }
		\label{A_+}
In this section, we consider the case of permutations of $\Z^2$ restricted by the set $A_+=\mathset{(0,\pm 1),(\pm 1,0)}$, presented in Example~\ref{ex:GenToZd}. In that case, the corresponding graph described in Theorem~\ref{th:GenCan} (the general correspondence) is a $\Z^2$-periodic bipartite graph, but in this $\Z^2$-periodic presentation it has intersecting edges (see Figure~\ref{fig:Gen_A+}). Thus, we cannot use the results from the theory of $\Z^2$-periodic bipartite planar graphs as in the case of $\Omega(A_L)$. Fortunately, the graph $G_{A_+}$ (see Figure~\ref{fig:Graph_PL}) is $\Z^2$-periodic, bipartite, and planar (when we think of it as an undirected graph by merging the two edges with opposite directions between adjacent vertices). Using the alternative correspondence to perfect matchings, we have that restricted permutation of $G_{A_+}$ correspond to pairs of perfect matchings of the square lattice in $\Z^2$. 
\begin{figure}
 \centering
  \includegraphics[width=80mm, scale=0.5]{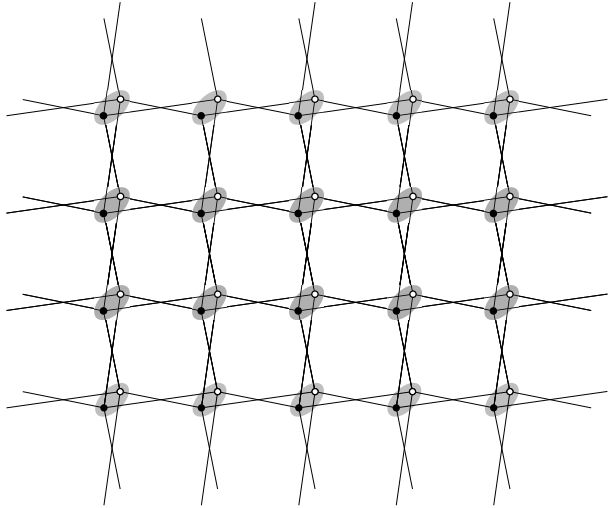}
  \caption{The corresponding graph for $G_{A_+}$ from Theorem~\ref{th:GenCan}. Black points represent vertices of the form $(I,n)$ and white points represent vertices of the form $(O,n)$.    }
  \label{fig:Gen_A+}
\end{figure}

 The problem of finding the exponential growth rate of perfect matchings (also called dimer coverings) of the square lattice , also known as the square lattice dimer problem or domino tiling problem, was studied thoroughly in the last century  (e.g., see \cite{Fis61, Kas61, ChoKenPro01, Fis66, TemFis60}). 

We show that $\Omega(A_+)$ is conjugated to the Cartesian product of the SFT of dimer coverings of the square lattice. We use some well known results regarding the square lattice dimer model in order to find the (topological, periodic and closed) entropy. Finally, we discuss methods for counting patterns in polynomial-time.

Given $M\in \PM(\Z^2)$ and $n\in \Z^2$, there exists a unique element  $m\in \Z^2$ such that $\mathset{n,m}\in M$, we denote such an $m$ by $M(n)$. Furthermore, from the definition of the square lattice $M(n)\in n+A_+$. For all $n\in \Z^2$, we define $\omega_M(n)\eqdef M(n)-n$. When we identify each element $M\in \PM(\Z^2)$ with $\omega_M\in A_+^{\Z^2}$, we get an embedding of $\PM(\Z^2)$ in $A_+^{\Z^2}$. This embedding of $\PM(\Z^2)$ in $A_+^{\Z^2}$ is an SFT, which we denote by $\Omega_D$.
We consider the topological space $\Omega_D^2\eqdef \Omega_D\times \Omega_D$, equipped with the product topology, which is compact. With the continuous group action of $\Z^2$ on $\Omega_D^2$ given in Remark~\ref{re:bipartiso}, the pair $(\Omega_D^2,\Z^2)$ forms a topological dynamical system. Applying  Corollary ~\ref{cor:undirected-PM}, we have that $\Omega(A_+)$ and $\Omega_D^2$ are topologically conjugate. 
\begin{corollary}
\label{cor:PlusEnt}
\begin{align*}
\topent(\Omega(A_+))=2 \topent \parenv{\Omega_D}
\end{align*}
\end{corollary}
\begin{proof}
We recall that the entropy of a direct product is the sum of entropies. Thus, $\topent(\Omega_D^2)=2\topent(\Omega_D)$. Since $\Omega(A_+)$ and $\Omega_D^2$ are conjugate they have the same entropy and \[ \topent(\Omega(A_+))=\topent(\Omega_D^2)= 2 \topent \parenv{\Omega_D}. \]
\end{proof}

The equivalence between the double dimer model and $\Omega(A_+)$ we have just proved may also be used in order to find the periodic and closed entropy of $\Omega(A_+)$.  For $n\in \N$, let $L_{S,n}=(V_n,E_n)$ be the $n\times n$ square sub-graph of the square lattice, that is 
\[ V_n\eqdef [n]\times [n], \ E_n\eqdef \mathset{\mathset{v_0,v_1}\in ([n]\times [n])^2 ~:~ \norm{v_0-v_1}_1=1 } .\]
Let $L_{S,n}^T=(V_n^T,E_n^T)$ be the $n\times n$ square lattice on the torus, 
\[V_n^T\eqdef [n]\times [n], \ E_n^T\eqdef \mathset{\mathset{v_0,v_1}\in ([n]\times [n])^2 ~:~ \norm{(v_0-v_1)\bmod (n,n)}_1=1 }. \]
We note that $L_{S,n}$ is a sub graph of $L^T_{S,n}$ and therefore $\PM(L_{S,n})$ is a subset of $L_{S,n}^T.$

In Kasteleyn's original work \cite{Kas61}, an exact formula for $\abs{\PM(L_{S,2n}^T)}$ was given, which was later used to show that
\[ \lim_{n\to \infty} \frac{\log\abs{\PM(L_{S,2n}^T)}}{4n^2}=\frac{1}{4}\cdot \intop_0^1 \intop_0^1\log(4-2\cos(2\pi x) -2\cos (2\pi y))dxdy.\]
It is also shown in \cite{Kas61}, that the exponential growth rate of $\abs{\PM(L_{S,n})}$ is the same as  $\abs{\PM(L_{S,n}^T)}$, That is 
\[ \lim_{n\to \infty} \frac{\log\abs{\PM(L_{S,2n}^T)}}{4n^2}= \lim_{n\to \infty} \frac{\log\abs{\PM(L_{S,2n})}}{4n^2}.\]

\begin{proposition}
\label{prop:A_+perCloed}
\[ \topent_c(A_+)=\topent_p(\Omega(A_+))=2 \lim_{n\to \infty} \frac{\log\abs{\PM(L_{2n}^T)}}{4n^2} \]
\end{proposition}

\begin{proof}
We saw that a permutation of $\Z^2$ restricted by $A_+$ corresponds to a pair of dimer coverings of $\Z^2$. For $n\in \N$, we consider the restriction of the bijection $\Phi: \Omega(A_+) \to \PM(\Z^2)\times \PM(\Z^2)$ to $\fix_{n\Z^2}(\Omega(A_+))$. Because this bijection is equivariant, for each $n$ it induces a bijection between $\fix_{n\Z^2}(\Omega(A_+))$ and $\fix_{n\Z^2}(\PM(\Z^2) \times \PM(\Z^2))=\fix_{n\Z^2}(\PM(\Z^2))\times \fix_{n\Z^2}(\PM(\Z^2))$. Any periodic perfect matching $M\in \fix_{n\Z^2}(\PM(\Z^2))$ is naturally identified with a perfect matching of the $n\times n$-torus   $M'\in \PM(L_{S,n}^T)$ defined by \[ M'=\mathset{\mathset{u \bmod (n,n),v \bmod (n,n)}: \mathset{u,v}\in M}.\] This shows that $\abs{\fix_{n\Z^2}(\Omega(A_+))}=\abs{\PM(L_{S,n}^T)}^2$.

We recall that the set of $A_+$-restricted permutations of the square $[n]\times [n]$, $\Pi_{(n,n)}(A_+)$, is identified with a subset of $\fix_{n\Z^2}(\Omega(A_+))$ (see Section \ref{CHA:preliminaries}, after Remark \ref{re:1}). Applying the bijection $\Phi$ to this subset, $\Phi(\Pi_{(n,n)}(A_+))\subseteq \fix_{n\Z^2}(\PM(\Z^2))\times \fix_{n\Z^2}(\PM(\Z^2)) $ is identified with a subset of $\PM(L_{S,n}^T)\times \PM(L_{S,n}^T)$. It is easy to verify that this subset is exactly $\PM(L_{S,n})\times \PM(L_{S,n})$ and thus $\abs{\Pi_{(n,n)}(A_+)}=\abs{\PM(L_{S,n})}^2$. All together we have, 
\[\topent_c(A_+)=\limsup_{n\to\infty}\frac{\log \abs{\Pi_{(n,n)}(A_+)}}{n^2}=2\limsup_{n\to\infty}\frac{\log \abs{\PM(L_{S,n})}}{n^2}\\
\] 
and 
\[ \topent_p(\Omega(A_+)= \limsup_{n\to\infty}\frac{\log \abs{\fix_{n\Z^2}(\Omega(A_+))}}{n^2} =2\limsup_{n\to\infty}\frac{\log \abs{\PM(L_{S,n}^T)}}{n^2}. \]
For odd $n$, since $\abs{L_{S,n}}=\abs{L_{S,n}}=n^2$ is odd, we have that $\PM(L_{S,n}^T)=\PM(L_{S,n})=\emptyset$. Hence, 
\begin{align*}
\limsup_{n\to\infty}\frac{\log \abs{\PM(L_{S,n})}}{n^2}&=
\lim_{n\to\infty}\frac{\log \abs{\PM(L_{S,2n})}}{4n^2}\\ 
&=\lim_{n\to\infty}\frac{\log \abs{\PM(L_{S,2n}^T)}}{4n^2}=\limsup_{n\to\infty}\frac{\log \abs{\PM(L_{S,n}^T)}}{n^2}.
\end{align*}
This completes the proof.
\end{proof}

In their work, Meyerovitch and Chandgotia  \cite{MeyCha19} explain the well known result 
\[ \topent (\Omega_D)=\lim_{n\to \infty} \frac{\log\abs{\PM(L_{S,2n}^T)}}{4n^2}. \]
In their proof, they use a principle called reflection positivity, relying on the symmetry of the uniform measure on perfect matchings,  with respect to reflection along
some hyperplanes.

Combining this result with Proposition \ref{prop:A_+perCloed} and Corollary \ref{cor:PlusEnt}, we obtain:
\begin{theorem}
\label{th:Sol+}
\[ \topent_c(A_+)=\topent_p(\Omega(A_+))=\topent(\Omega(A_+))=\frac{1}{2}\cdot \intop_0^1 \intop_0^1\log(4-2\cos(2\pi x) -2\cos (2\pi y))dxdy. \]
\end{theorem}  

\begin{remark}
Kasteleyn provided an exact formula for $\abs{\PM(L_{S,2n}^T)}$ \cite{Kas61}. This formula can be used in order to compute the exact number of toral permutations restricted by $A_+$, as $\abs{\fix_{n\Z^2}(\Omega(A_+))}=\abs{\PM(L_{S,2n}^T)}^2$. In order to get a more complete picture, we want to be able to compute the exact number of patterns in $B_n(\Omega(A_+))$ and in $\Pi_n(A_+)$ (permutations of $[n]$ restricted by $A_+$), for any given $n\in\N^d$. We already know that 
\[ \abs{\Pi_n(A_+)}=\abs{\PM(L_{S,n})}^2. \]
Since $L_{S,n}$ is a finite planar graph, using Kasteleyn's method, $\abs{\PM(L_{S,n})}$ is computable in polynomial time. Therefore, $\abs{\Pi_n(A_+)}$ is computable in polynomial time as well.  
For the computation of $\abs{B_n(\Omega(A_+))}$, we consider the conjugation map between $\Omega(A_+)$ and $\Omega_D^2$ from Theorem \ref{th:PerToPM}. It is easy to show that the restricting of this map to patterns in $B_n(\Omega(A_+))$ we get a bijection between $B_n(\Omega(A_+))$ and $B_n(\Omega_D)\times B_n(\Omega_D)$. Therefore, we have  $\abs{B_n(\Omega(A_+))}=\abs{B_n(\Omega_D)}^2$. Elements of $B_n(\Omega_D)$ represent perfect coverings of $[n]$ in $\Z^2$. By Theorem~\ref{prop:GenAlg}, $\abs{\PC([n],\Z^2)}$ is computable in polynomial time, and therefore 
$\abs{B_n(\Omega(A_+))}=\abs{\PC([n],\Z^2)}^2$ as well.
\end{remark}

\section{Entropy}
\label{CHA:Entropy}
In this section, we investigate the entropy of dynamical systems defined by $\Z^d$-permutations, restricted by some finite set $A\subseteq \Z^d$. We start by proving some basic properties of the topological entropy for such SFTs and use them in order to find the topological entropy whenever $A$ is composed of $3$ elements which are not contained in a line. We discuss the topic of global and local admissibility of patterns.  In the last part, we review two related models of injective and surjective restricted functions of graphs.

	\subsection{Properties}
	\label{EntProperties}
We show that the entropy of $\Z^d$-permutations, restricted by some finite set $A$, is invariant under the operation of an injective affine transformation on $A$. Furthermore, we prove the conjugacy of $\Omega(TA)$ and $\Omega(A)^{[\Z^d:T\Z^d]}$ in the case where $T\in \End(\Z^d)$  (where $\End(\Z^d)$  denotes the sets of $\Z^d$-endomorphisms).

\begin{fact}
\label{fact:ShiftInv}
\cite[Proposition 1.1]{SchStr17} Let $d\geq 1$, and $A\subseteq\mathbb{Z}^{d}$ be a finite set. For any $b\in \Z^d$, $\Omega(A)$ and $\Omega(A+b)$ are topologically conjugate (where $A+b$ denotes $\sigma_b(A)$). 
\end{fact}

\begin{proposition}
\label{prop:MatEnt}
Let $d\geq 1$ be an integer and $A\subseteq \Z^2$ be a finite set. For any $\Z^d$-endomorphism, $M\in \End(\Z^d)$, $\Omega(MA)$ with the action of $M\Z^d$ induced from the regular shift action of $\Z^d$ is isomorphic to $\Omega(A)^{[ \Z^d : M\Z^d ]}$ with the $\Z^d$ action defined by $n(\omega_1,\dots,\omega_d)=(\sigma_n(\omega_1),\dots,\sigma_n(\omega_d))$. As a consequence, $ \topent(\Omega(A))=\topent(\Omega(M A))$.
\end{proposition}
\begin{proof}
Denote the index of the $M\Z^d$ inside $\Z^d$ by $k$. Let $H_1,\dots,H_k$ be the cosets of $M\Z^d$ and $v_1,\dots,v_k$ be representing vectors, i.e., $v_i\in H_i$ for all $i$.

Given $\omega\in \Omega(M A)\subseteq (M A)^{\Z^d}$, we define $\Phi(\omega)\eqdef (\omega_1,\omega_2,\dots, \omega_k)$, where $\omega_i\eqdef M^{-1}\circ \omega \circ \sigma_{v_i} \circ  M$, and $\sigma_{v_i}$ is the regular shift by $v_i$ in $\Z^d$. That is,
\[ \omega_i(n) =M^{-1} \omega(Mn+v_i). \]
Clearly, $\omega(Mn+v_i)\in M A$ and therefore $\omega_i(n) =M^{-1} \omega(Mn+v_i)\in A$ and $\omega_i\in A^{\Z^d}$ for all $1\leq i \leq k$.  First we show that $\omega_j\in \Omega(A)$ for all $j$. That is, $\pi_{\omega_j}$ is a permutation of $\Z^d$  (where as usual, $\pi_{\omega_j}$ is defined by $\pi_{\omega_j}(n)=n+\omega_j(n)$).

 Injectivity -- let $n,n'\in \Z^d$, since $\pi_\omega$ is a permutation of $\Z^d$ we have:
\begin{align*}
&\pi_{\omega_j}(n)=\pi_{\omega_j}(n')\\
&\iff n+\omega_j(n)=n'+\omega_j(n')\\
&\iff M( n+\omega_j(n))=M(n'+\omega_j(n'))\\
&\iff M(n)+v_j+ M(\omega_j(n))= M(n')+v_j+ M(\omega_j(n'))\\
&\iff M(n)+v_j+ M\parenv{M^{-1}(\omega(M(n)+v_j))}= M(n')+v_j+ M\parenv{M^{-1}(\omega(M(n')+v_j))}\\
&\iff \underset{\pi_\omega(M(n)+v_j)}{\underbrace{M(n)+v_j+\omega(M(n)+v_j)}}=\underset{\pi_\omega(M(n')+v_j)}{\underbrace{M(n')+v_j+\omega(M(n')+v_j})}\\
&\iff M(n')+v_j=M(n')+v_j\iff n=n'.
\end{align*}

 Surjectivity -- let $n\in \Z^d$. There exists $m\in \Z^d$ such that $\pi_{\omega}(m)=M(n)+ v_j$. Since $\pi_\omega$ is restricted by $M   A$, $m$ belongs to the same coset as $M(n)+v_j$, which is $H_j$.
Thus, $m$ is of the form $M(m')+v_j$ for some $m'\in \Z^d$. We have 
\begin{align*}
\pi_{\omega_j}(m')&=m'+\omega_j(m')=m'+M^{-1}(\omega(M(m')+v_j))\\
&=M^{-1}(\underset{m}{\underbrace{M(m' ) +v_j}} +\underset{\omega(m)}{\underbrace{\omega(M(m')+v_j)}})-	M^{-1}(v_j)=M^{-1}(\underset{\pi_\omega(m)}{\underbrace{m+\omega(m)}})- M^{-1}(v_j)\\
&=M^{-1}(M(n) + v_j)-M^{-1}(v_j)=M^{-1}(M(n))=n.
\end{align*}

We claim that $\Phi$ is invertible. For $n\in \Z^d$, we define $j_n$ to be the index of the coset for which $n\in H_{j_n}$. Given $\omega_1,\dots \omega_{k}\in \Omega(A)$ and $n\in \Z^d$, we define
\[ \omega(n)=M(\omega_{j_n}(M^{-1}(n-v_{j_n})))\in M  A ,\]
and $\Psi(\omega_0,\dots,\omega_{k-1})=\omega$. We observe that for any $j$, $\omega_j$ defines the restriction of  $\pi_\omega$ to the coset $H_j$. We may repeat the same arguments used in the first part of the proof (in reversed order) to show that this restriction is a permutation of the coset $H_j$. Thus $\pi_\omega$ is a permutation of $\Z^d$ and $\omega \in \Omega(M  A)$. It is easy to verify that $\Psi$ is exactly the inverse function of $\Psi$, and thus $\Psi$ and $\Phi$ are bijections.

It remains to to prove that the actions of $M\Z^d$ and $\Z^d$ on $\Omega(MA)$ and $\Omega(A)$ commutes and that $\Phi$ is a homeomorphism. For $k\in \Z^d$, 
\begin{align*}
\Phi_j(\sigma_{Mk} \omega)(n)&=M^{-1}(\sigma_{Mk}\omega)(Mn+v_j)=M^{-1}(\omega)(M(n+k)+v_j)\\
&=\Phi_j(\omega)(n+k)=(\sigma_k\Phi(\omega))(n).
\end{align*} 
We note that $\Phi$ and $\Phi^{-1}$ depend on finitely many coordinates. Hence, $\Phi$ is an homeomorphism. 

By the subgroup entropy formula \cite[Proposition 13.1]{Schm12}, the topological entropy of $ \topent(\Omega(MA)$ with the $\Z^d$ action given by $n\omega=\sigma_{Mn}\omega$ is  $k$ times the topological entropy of $\Omega(MA)$ with the usual $\Z^d$ action by shifts. Since the entropy of the direct product is the sum of entropies,  we have $\topent(\Omega(A)^k)=k \topent(\Omega(A))$ and the desired equality holds.

\end{proof}
\begin{theorem}
\label{th:AffInvEnt}
Let $A\subseteq \Z^d$ be a finite set. If $T$ is an invertible affine transformation, that is $Tx=Mx+b$ where $b\in \Z^d$ and $M$ is a $d\times d$ invertible matrix (over $\R$), then
\[ \topent (\Omega(A))=\topent (\Omega (T(A)) ).\] 
\end{theorem}
\begin{proof}
Follows directly from Propositions~\ref{fact:ShiftInv} and \ref{prop:MatEnt}.
\end{proof}

\begin{definition}
\label{def:Aff}
Let $A=\mathset{x_1,x_2,\dots,x_n}$ be a finite set of points, contained in some vector space $V$ over $\R$. The affine dimension of $A$, denoted by $\dim_{\text{aff}}(A)$, is defined to be the dimension of the vector space $V_A$, where 
\[ V_A\eqdef \mathset{\sum_{i=1}^n \alpha_i x_i~:~\alpha_1,\alpha_2,\dots, \alpha_d\in \R \such
\sum_{i=1}^n \alpha_i =0}.\]
We say that $A$ has full affine dimension if $\dim_{\text{aff}}(A)+1=\abs{A}$ and that the vectors composing $A$ are affinely independent.
\end{definition}

\begin{theorem}
\label{th:AffEqCEnt}
Let $d\geq 1$ and $A,B\subseteq \Z^d$ be finite sets with full affine dimension such that $\abs{B}=\abs{A}=d'\leq d+1$.
Then, $\topent(\Omega(A))=\topent(\Omega(B))$. Furthermore, If $d'=3$
\[ \topent(\Omega(A))=\topent(\Omega(B))=\frac{1}{4\pi^2}\intop_0^{2\pi}\intop_0^{2\pi}\log \abs{1+e^{ix}+e^{iy}}dxdy. \]
\end{theorem}

\begin{proof}

By the previous theorem it is enough to show that there exist two invertible integral affine maps $T_1$ and $T_2$ and a set $C_{d',d}\subseteq\Z^d$ such that $T_1(C_{d',d})=A $ and $T_2(C_{d',d})=B$. By applying a translation to $A$ and $B$, we can assume that both $A$ and $B$ contain 0.  Then it is enough to prove that there exist endomorphisms $M_1$ and $M_2$ of $\Z^d $ and a set $C_{d',d}\subseteq\Z^d$ such that  $T_1(C_{d',d})=A $ and $T_2(C_{d',d})=B$. Equivalently, if $A' = A \setminus \mathset{0}$ and $B' = B \setminus \mathset{0}$, then both $A'$ and $B'$ are linearly independent sets. Complete each of them to a basis of $\mathbb{Q}^d$ composed from integer vectors, $A''=(a_1,\dots,a_d)$ and $B''=(b_1,\dots,b_d)$, where $a_1,\dots, a_{d'-1}$ and $b_1,\dots, b_{d'-1}$ are the vectors in $A'$ and $B'$ respectively
 . Let $C_{d,d'}\eqdef \mathset{0,e_1,e_2,\dots,e_{d'-1}}$ where $e_i$ is the standard $i$'th unit vector. Let $M_1$ and $M_2$ be the matrices whose rows are the elements of $A''$ and $B''$ respectively. Clearly, $M_1$ and $M_2$ map the standard basis $e_1,\dots,e_d$ to $A''$ and $B''$ respectively and in particular $M_1(C_{d',d})=A$ and $M_2(C_{d',d})=B$ as desired.

If $d=2$ and $d'=3$ we note that $C_{2,3}=\mathset{(0,0),(0,1),(1,0)}$, which in the notation of Section~\ref{A_L}, is the set $A_L$. By Theorem~\ref{th:SolL}, 
\[  \topent(\Omega(A))=\topent(\Omega(A_L))=\frac{1}{4\pi^2}\intop_0^{2\pi}\intop_0^{2\pi}\log\abs{1+e^{ix}+e^{iy}}dxdy.\]
If $d'=3$ and $d\geq 3$, we note that $C_{d,3}=\mathset{0,e_1,e_2}$ and $\Omega(C_{d,3})$ is in fact 2-dimensional as for any permutation $\pi\in \Omega(C_{d,3)}$ and an index $m\in \Z^{d-2}$, the restriction of $\pi$ to $\Z^2\times\mathset{m}$ is a restricted permutation of $\Z^2\times\mathset{m}$. Thus, for any $n=(n_1,n_2,\dots,n_d)$ we have $\abs{B_{n}(\Omega(C_{d,3}))}=\abs{B_{n_1,n_2}(\Omega(C_{2,3})}^{\prod_{i=3}^d n_i}$ and therefore $\topent(\Omega(C_{2,3}))=\topent(\Omega(C_{d,3}))$. This completes the proof.
\end{proof}

\begin{corollary}
\label{cor:UnifBound}
For $A\subseteq \Z^d$ which is not contained in a line 
\[  \topent(\Omega(A))\geq \frac{1}{4\pi^2}\intop_0^{2\pi}\intop_0^{2\pi}\log\abs{1+e^{ix}+e^{iy}}dxdy. \]
\end{corollary}
\begin{proof}
If $A$  is not contained in a line, there exists $\mathset{a_1,a_2,a_3}=A'\subseteq A$ with full affine dimension. Clearly $\topent(\Omega(A))\geq \topent(\Omega(A')) $ as  $\Omega(A')\subseteq \Omega(A)$. By Theorem~\ref{th:AffEqCEnt}
\[\topent(\Omega(A))\geq \topent(\Omega(A')) = \frac{1}{4\pi^2}\intop_0^{2\pi}\intop_0^{2\pi}\log \abs{1+e^{ix}+e^{iy}}dxdy. \] 
\end{proof}

			\subsection{Local and Global Admissibility}
	\label{LocGlob}
Given an $SFT$, $\Omega\subseteq \Sigma^{\Z^d}$, a finite set $U\subseteq \Sigma^d$, and a pattern $v\in \Sigma^U$, a natural question is whether this pattern is globally admissible, i.e., whether there exists $\omega\in \Omega$ such that the restriction of $\omega$ to $U$ is the pattern $v$. Generally, this question does not have a simple answer. It is proved in \cite{Rob71} that in the general case, it is not decidable whether a finite pattern is globally admissible, i.e., there is no algorithm that can decide  whether a finite pattern is globally admissible or not.

If $\Omega$ is defined by the set of forbidden patterns $F$,  a necessary condition for global admissibility is local admissibility. We say that a pattern $v\in \Sigma^U$ is locally admissible if it does not contain any of the forbidden patterns in $F$. That is, for any forbidden pattern $p \in F \cap \Sigma^{U'}$ and $n\in \Z^d$ such that $U'\subseteq \sigma_n(U)$, $\parenv{\sigma_n(v)}(U')\neq p$. 
Clearly, if a pattern is globally admissible, it is also locally admissible. However, local admissibility does not imply global admissibility. Example~\ref{Ex:LocAd} below provides a pattern which is locally admissible but not globally admissible in the context of restricted permutations.

For a finite restricting set $A\subseteq \Z^d$, a pattern $v\in A^U$ is identified with a function $f_v: U\to U+A$, defined by $f_v(n)=n+v(n)$. The pattern $v$ is globally admissible if it is the restriction of some $\omega\in \Omega(A)$. For such $\omega\in \Omega(A)$, we have that $f_v$ is the restriction of the permutation $\pi_\omega\in \Omega(A)$ to the set $U$. Thus, global admissibility of $v$ is equivalent to the existence of a permutation $\pi\in S(\Z^d)$, restricted by $A$, extending $f_v$.

In Proposition~\ref{prop:NecessExt} we present a description of a necessary locally checkable  condition for global admissibility. We later discuss the the relation of this condition to the concept of local admissibility. In Proposition~\ref{prop:compli2} we show that this  locally checkable necessary condition for global admissibility is sufficient when considering rectangular patterns in two cases of restricting sets. We use these results for counting rectangular patterns in in Section~\ref{PM_char}.

\begin{definition}
Let $A,U\subseteq \Z^d$ be some finite sets. The boundary of $U$ with respect to $A$, denoted by $\partial (U,A)$, is defined to be the set of all indices $u\in U$ for which $u-A\not\subseteq U$. The interior of $U$ with respect to $A$ is defined to be $\Int(U,A)\eqdef U\setminus \partial (U,A)$
\end{definition}

\begin{proposition}
\label{prop:NecessExt}
Let $A\subseteq \Z^d$ be a finite non-empty restricting set and $U\subseteq \Z^d$ be some set. If a pattern $v\in A^U$ is globally admissible, then $f_v:U\to U+A$ defined by $f(n)=n+v(n)$ is injective and $\Int(U,A)\subseteq \ima(f_v)$. 
\end{proposition}

\begin{proof}
Assume that $v$ is globally admissible, and let $\omega\in \Omega(A)$ be such that $\omega([n])=v$. We note that $f_v$ is the restriction of $\pi_\omega$ to $U$, where $\pi_\omega:\Z^2\to \Z^2$ is the permutation defined by $\pi_\omega(m)=m+\omega(m)$. Clearly, $f_v$ is injective since $\pi_\omega$ is injective (as a permutation). Since $\pi_\omega$ is surjective, $\Int(U,A)\subseteq \ima( \pi_\omega)$. On the other hand, by the definition if $\Int(U,A)$ 
\[ \pi_\omega^{-1}(\Int(U,A)) \subseteq \Int(U,A)- A \subseteq U,\]
as $\pi_\omega$ is restricted by $A$.  Thus, $\Int(U,A)\subseteq \pi_\omega(U)=\ima(f_v).$
\end{proof}

The locally checkable condition  presented in Proposition~\ref{prop:NecessExt} are necessary for global admissibility. The following example shows that in genreal, they are not sufficient.

\begin{example}
\label{Ex:LocAd}
Consider the restricting set $A_+=\mathset{(0,\pm 1),(\pm 1,0)}$ reviewed in Section~\ref{A_+} and the set $U\eqdef [5]\times [3]\setminus \mathset{(1,2)}$. It is easy to see that restricted function $f:U\to U+A_+$ presented  in Figure~\ref{fig:LocAdmiss} is injective. Furthermore, $\Int(U,A_+)$ is an empty set and therefore it is contained in the image of $f$ in a trivial way. So $f$ is locally admissible. Assume to the contrary that there exists $\pi\in\Omega(A_+))$ extending $f$. We note that for both $(1,1)$ and $(3,1)$, the only possible pre-image is $(2,1)$. Hence, $\pi$ cannot be surjective which is a contradiction. This shows that $f$ is not globally admissible.

\begin{figure}
 \centering
  \
  \includegraphics[width=65mm, scale=0.1]{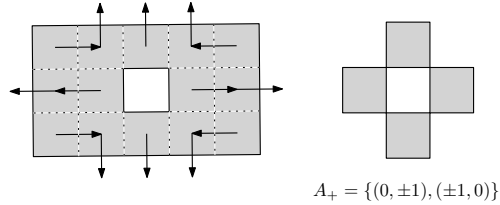}
  \caption{A locally admissible pattern which is not globally admissible where the restricting set is $A_+$.  }
  \label{fig:LocAdmiss}
\end{figure}
\end{example}

Having a locally checkable necessary condition for global admissibility of patterns, it is natural to ask whether it is equivalent to local admissibility. The answer to this question is yes and no. We recall that our definition for local admissibility of patterns in SFTs uses the finite set of forbidden patterns which defines the SFT (see the beginning of Section \ref{LocGlob}). Any SFT can equivalently be defined by infinitely many sets of forbidden patterns and therefore the conditions of local admissibility does not depend only on the SFT itself. One natural set of forbidden patterns which defines $\Omega(A_+)$ is  $F_{A_+}\eqdef \mathset{w\in A_+^{-A_+} ~:~ \abs{\mathset{n\in -A_+ ~:~ w(n)+n=0}}\neq 1}\subseteq A_+^{-A_+}$. With respect to this set, the pattern $v=\sparenv{(1,0),(0,0)}\in \mathset A^{\mathset{(0,0),(1,0)}}$ is locally admissible since for every $n\in \Z^2$, $\sigma_n(-A_+)\not\subseteq \mathset{(0,0),(1,0)}$ so the condition for local admissibility is satisfied in a trivial way. We note that $v$ does not satisfy the locally checkable condition from Proposition \ref{prop:NecessExt} as $f_v$ is not injective. This shows that the two conditions are not equivalent. On the other hand the set \[F_{A_+}'\eqdef F_{A_+}\bigcup \mathset{v\in A^{\mathset{-a,-b}} ~:~ a,b\in A, a\neq b, v(a)+a=v(b)+b }\]
also defines $\Omega(A)$ and it is an easy exercise to check that with this set of forbidden patterns, the the necessary locally checkable condition and local admissibility are equivalent.

\begin{proposition}
\label{prop:compli2}
For $A\in \mathset{A_\oplus, A_L}$, where $A_L$ is defined in Section~\ref{CHA:preliminaries} and $A_\oplus\eqdef \mathset{(0,0),(0,\pm 1), (\pm 1,0)}$, let $(n_1,n_2)=n\in \N^2$ and $v\in A^{[n]}$ be a rectangular pattern. Then $v$ is globally admissible  if and only if $f_v:[n]\to [n]+A$ defined by $f_v(m)=m+v(m)$ is injective and $\Int([n],A)\subseteq \ima(f_v)$.
\end{proposition}

\begin{proof}
We will show the proof for $A=A_\oplus$, the proof in the case that $A=A_L$ follows a similar idea. In this proof we assume that $n_1,n_2>3$, so the rectangle is sufficiently large for our needs. The claim is true for smaller $n_1$ and $n_2$ as well and not it is not difficult to check by hand. 
The first direction (global admissibility implies injective $f_v$ and $\Int(U,A)\subseteq \ima(f_v)$) is is true by Proposition~\ref{prop:NecessExt}. 

For the other direction, we first note that $\Int([n],A_\oplus)=[1,n_1-1]\times [1,n_2-1]$. 
Assume that $\pi_v$ is injective and $[1,n_1-1]\times [1,n_2-1]\subseteq \ima(\pi_v)$, we need to find a restricted permutation $\pi\in (\Omega(A_\oplus))$ such that the restriction $\pi([n])$ is $\pi_v$. Consider the sets $O_v=\ima(\pi_v)\setminus [n]$, and $H_v=[n]\setminus \ima(\pi_v)$.

We observe that $v\in A_\oplus^{[n]}$, which implies that $\ima(\pi_v)\subseteq [n]+A_\oplus$. Thus, $O_v\subseteq ([n]+A_\oplus)\setminus [n]$. We denote $([n]+A_\oplus)\setminus [n]$ by $\partial^{A_\oplus} [n]$. By the assumption, $\pi_v$ satisfies the second condition. Hence, $H_v\subseteq \partial([n],A_\oplus)=[n]\setminus [[1,n_1-1]\times [1,n_2-1]]$. We observe that for $m\in \partial^{A_\oplus} [n] $, there exists a unique vector $e_m\in {\mathset{(0,\pm 1), (\pm 1,0)}}$ such that $m+e_m\in [n]$. Similarly, for $m\in \partial ([n],A_\oplus)$, there exists a unique vector $e_m\in {\mathset{(0,\pm 1), (\pm 1,0)}}$ such that $m+e_m\notin [n]$. We now define $\pi$ on $\Z^2\setminus [n]$. 
\begin{itemize}
\item For $m\in O_v$ such that $m+e_m\in H_v$, define $\pi(m)=m+e_m$.
\item For $m\in O_v$ such that $m+e_m\notin H_v$, we define $\pi(m)=m-e_m$. Furthermore, for all $k\in \N$ we define $\pi(m-k\cdot e_m)= m-(k+1)\cdot e_m$.
\item For $m\in H_v$ such that $\pi(m+e_m)\notin O_v$, we define $\pi(m+k\cdot e_m)=m+(k-1)\cdot e_m$ for all $k\in \N$.
\item For any index $m\in \Z^2\setminus [n]$ not defined in the first three items, we define $\pi(m)=m$.
\end{itemize} 
Clearly, $\pi$ is restricted by $A_\oplus$ as $e_m\in A_\oplus$ for any $m\in H_v \cup O_v$. It may be verified that $\pi$ is bijective.
\begin{figure}
 \centering
  \
  \includegraphics[width=100mm, scale=0.1]{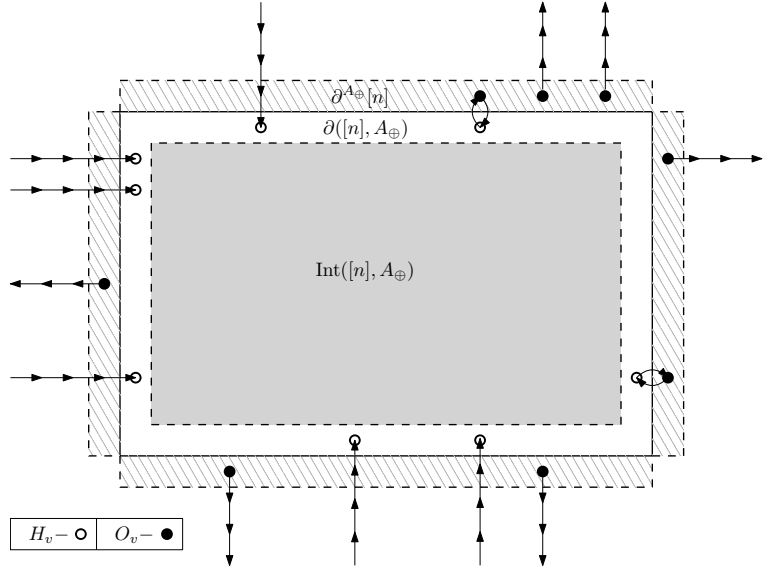}
  \caption{The extension of $\pi_v$ to a $\pi\in \Omega(A_\oplus)$. }
  \label{fig:Ext}
\end{figure}
\end{proof}
See Figure~\ref{fig:Ext} for a demonstration of the procedure of defining $\pi$.

			\subsection{The Entropy of Injective and Surjective Functions}
	\label{InjSurj}
So far, we have explored restricted permutations of graphs which are bijective functions. In this part of the work we examine the related models of restricted injective and surjective functions on graphs. We will show that under similar assumptions as in the case of permutations, the spaces of restricted injective and surjective  functions also have the structure of a topological dynamical system. Finally, we examine the entropy of restricted injective/surjective functions on $\Z^d$, compared to the entropy of restricted permutations.

Let $G=(V,E)$ be some locally finite and countable directed graph. Similarly to  Section~\ref{CHA:preliminaries}, a function $f:V\to V$ is said to be restricted by $G$ if $(v,f(v))\in E$ for all $v\in V$. We define the spaces of restricted injective and surjective functions of $G$ to be 
\[ \Omega_I(G)\eqdef \mathset{\varphi:V\to V~:~ \varphi \text{ is injective and restricted by }G},\]
and
\[ \Omega_S(G)\eqdef \mathset{\varphi:V\to V~:~\varphi \text{ is surjective and restricted by }G},\]
respectively.

The spaces $\Omega_I(G)$ and $\Omega_S(G)$ are compact topological spaces, when equipped with the product topology (when $V$ has the discrete topology).  If $\Gamma$ is a group acting on $G$ by graph isomorphisms, it induces a homeomorphic group action on $\Omega_I(G)$ and  $\Omega_S(G)$ by conjugation. This is proven in a similar fashion as in the case of restricted permutations (see Section~\ref{CHA:preliminaries}). Therefore, we will leave the details to the reader.

As in most of this work, we focus on the cases where $G=(\Z^d,E_A)$ for some finite $A\subseteq \Z^d$. In this case, we shorten the notation and use  $\Omega_I(A)$ for $\Omega_I(G_A)$ and  $\Omega_S(A)$ for $\Omega_S(G_A)$.

\begin{proposition}
For any finite non-empty $A\subseteq \Z^d$, $\Omega_I(A)$ and $\Omega_S(A)$ are SFTs, when we identify a restricted function  $\varphi$ with an element of $A^{\Z^d}$ by $\omega_\varphi(n)=\varphi(n)-n$. 
\end{proposition}

The proof of this statement is quite obvious and very similar to the case of permutations. Therefore, we leave it for the reader.

For any finite non-empty set $A\subseteq\Z^d$, we note that $\Omega(A)$ is exactly the intersection of $\Omega_I(A)$ and $\Omega_S(A)$. Furthermore, we observe that $\Omega(A)$ is strictly contained in both  $\Omega_I(A)$ and $\Omega_S(A)$. Thus, $\topent(\Omega(A))\leq \min\mathset{\topent(\Omega_I(A)),\topent(\Omega_S(A))}$. This gives rise to the natural question, whether the entropy of restricted permutations can be strictly smaller. Theorem~\ref{th:InjSurj} provides a negative answer to the above question.

\begin{definition}
Consider a measurable space $(X,\cF)$ and a group $\Gamma$ acting on $X$ by measurable transformations. A probability measure $\mu:\cF\to [0,1]$ is said to be invariant under the action of $\Gamma$ if for any measurable set $A\subseteq \cF$ and $\gamma\in \Gamma$ we have 
\[\mu\parenv{\gamma^{-1}(A)}=\mu(A). \]
 We define $M_{\Gamma}(X)$ to be the set of all probability measures on $X$ which are invariant under the action of $\Gamma$. A measure $\mu \in M_\Gamma(X)$ is called ergodic if it assigns invariant sets with $0$ or $1$. That is, $\gamma^{-1}A=A$ for all $\gamma\in \Gamma$  implies that $\mu(A)\in \mathset{0,1}$.
\end{definition}

\begin{proposition}
\label{prop:SuppMeasure}
If $\pi$ is a random function on $\Z^d$ restricted by $A$ which is either almost surely injective or almost surely surjective and its distribution is shift-invariant, then almost-surely $\pi$ is a permutation of $\Z^d$. Equivalently: The support of any shift-invariant measure on $\Omega_I(A)$ or $\Omega_S(A)$ is contained in $\Omega(A)$.
\end{proposition}

\begin{proof}
First, we prove the theorem for ergodic measures in $M_{\Z^d}(\Omega_I(A))$.  Recall that each element $\omega\in A^{\Z^d}$ is identified with a restricted function $\Z^d \to \Z^d$  by $f_\omega(n)\eqdef n+\omega(n)$. For $n\in \Z^d$ consider the function $P_n:A^{\Z^d}\to \N\cup \mathset{0}$ which assigns each $\omega\in A^{\Z^d}$ the number of pre-images of $n$. That is, $P_n(\omega)\eqdef \abs{f_\omega^{-1}(\mathset{n})}=\abs{\mathset{m\in \Z^d~:~ m+\omega(m)=n}}$.
Since $A$ is finite, there exists $M\in \N$ such that $A\subseteq [-M,M]^d$. For $(n_1,n_2,\dots,n_d)=n\in \N^d\setminus [-2M,2M]^d $ and $\omega\in A^{\Z^d}$ we examine the average of the functions $(P_m)_{m\in [n]}$, denoted by $A_n(\omega)$,
\begin{align*}
A_n(\omega)=\frac{\sum_{m\in [n]}P_m(\omega)}{\abs{[n]}}=\frac{\sum_{m\in [n]} \abs{f_\omega^{-1}(\mathset{m})}}{\abs{[n]}}=\frac{ \abs{f_\omega^{-1}([n])}}{\abs{[n]}}.
\end{align*}
Since $A$ is bounded in $[-M,M]^d$ and the movements of $f_\omega$ are restricted by $A$, we have  
\[[n_1-2M]\times \cdots \times [n_d-2M] \subseteq  f_\omega^{-1}([n]) \subseteq [n_1+2M]\times \cdots \times [n_d+2M].\] 
Thus, if we choose $n_k\eqdef (k,k,\dots,k)$ for $k\in \N$, for sufficiently large $k$, 
\[\frac{(k-2M)^d}{k^d}\leq A_{n_k}(\omega)\leq \frac{(k+2M)^d}{k^d}, \]
and in particular, $\lim_{k\to \infty} A_{n_k}(\omega)=1$.

Consider the measurable space $(\Omega_I,\mathcal{B_I})$, where $\mathcal{B_I}$ is the Borel $\Sigma$-algebra on $\Omega_I$. We note that for any $n\in \Z^d$ and $\omega\in \Omega_I$, $P_n(\omega)=P_0\circ \sigma_n(\omega)=P_0(n\omega)$, where $\sigma_n$ is the shift operation on $\Z^d$ and $n\omega$ denotes the group action of $n$ on $\omega$.

Consider the sequence of cubes, $([n_k])_{k=1}^{\infty}$, which is a Folner  sequence.
By the Pointwise Ergodic Theorem \cite{Lin01}, for an ergodic $\mu \in M_{\Z^d}(\Omega_I(A))$ the sequence $(A_{n_k})_k$ converges almost everywhere to $\E_\mu[P_0]$ (the expectation of $P_0$ with respect to the measure $\mu$). On the other hand, we saw that the sequence $(A_{n_k})_k$ converges pointwise to the constant function $1$. We conclude that $\E_\mu[P_0]=1$.

We observe that $P_0$ can take only the values $0$ and $1$ on $\Omega_I(A)$, as any function defined by an element in $\Omega_I(A)$ is injective. Hence, 
\[ 1=\E_\mu[P_0]=1\cdot \mu(P_0=1)+0\cdot \mu(P_0=0)=\mu(P_0=1).\]
Since $\mu$ is invariant under the action of $\Z^d$, for all  $n\in \Z^d$, \[\mu(P_0=1)=\mu(P_0\circ \sigma_n=1)=\mu(P_n=1). \]

We note that for $\omega\in \Omega_I(A)$, we have that $\omega\in \Omega(A)$ (i.e., $\omega$ represents a permutation), if and only if $f_\omega$ is also surjective. That is, any $n\in \N$ has a unique pre-image, which in the notation of this proof, is equivalent to $P_n(\omega)=1$ for all $n\in \Z^d$. We conclude that 
\[ \mu(\Omega(A))=\mu\parenv{\bigcap_{n\in \Z^d}\mathset{P_n=1}}=1.\] 

Now we turn to prove the claim for general $\mu\in \Omega_I(A)$. If $\mu$ is a convex combination of ergodic measures, then the claim follows immediately from the first case. For a general $\mu\in M_{\Z^d}(\Omega_I(A))$, By the ergodic decomposition theorem \cite[Theorem 4.8]{EinWar13}, $\mu$ is in the closed convex hull of the ergodic measures. That is, there is a sequence of measures $(\mu_n)_n\subseteq M_{\Z^d}(\Omega_I(A))$ which converges to  $\mu$ in the weak-* topology, and $\mu_n$ is a convex combination of a ergodic measures for each $n$. We obtain,
\[ \mu(\Omega(A))=\lim_{n\to\infty}\mu_n(\Omega(A))=1. \]

The proof for $\Omega_S(A)$ is very similar. Considering the restriction of the functions $(P_n)_{n\in \Z^d}$ to $\Omega_S(A)$, we observe that they can only take values greater than or equal to $1$ as for any element $\omega\in \Omega_S(A)$, $f_\omega$ is surjective. By similar arguments as in the previous case, for any ergodic invariant probability measure $\nu\in M_{\Z^d}(\Omega_S(A))$, we have $\E_\nu[P_n]=1$ for all $n\in \Z^d$. Since $P_n\geq 1$, we conclude that $\nu (P_n=1)=1$ for all $n$ and therefore $\nu (\Omega(A))=1$. For a non-ergodic measure, we continue in a similar fashion as in the proof of the claim for $\Omega_I(A)$. 
\end{proof}

\begin{theorem}
\label{th:InjSurj}
For any finite non-empty set  $A\subseteq\Z^d$, 
\[ \topent(\Omega(A))=\topent(\Omega_I(A))=\topent(\Omega_S(A)). \] 
\end{theorem}

\begin{proof}
By the variational principle \cite{Mis67}, 
\[ \topent(\Omega_I(A))=\max_{\mu\in M_{\Z^d}(\Omega_I(A))}H(\mu), \] 
where $H(\mu)$ is the measure theoretical entropy of $\mu$. Let $\mu_0\in M_{\Z^d}(\Omega_I(A))$ be a measure such that $ \topent(\Omega_I(A))=H(\mu_0)$. Proposition~\ref{prop:SuppMeasure} suggest that $\mu_0(\Omega(A))=1$, therefore the restriction of $\mu_0$ to the subspace $\Omega(A)$ is an invariant probability measure on $\Omega(A)$. Using the  variational principle once again, 
\[ \topent(\Omega(A))=\max_{\mu\in M_{\Z^d}(\Omega(A))}H(\mu)\geq  H(\mu_0)=\topent(\Omega_I(A)).\] 
The other direction of inequality is trivial as $\Omega(A)\subseteq \Omega_I(A)$, so equality holds. 
The proof for $\Omega_S(A)$ is exactly the same.

\end{proof}


\section{Further questions and comments}
We conclude with some comments and further questions:
\begin{enumerate}
\item Mahler measures and entropy: The Mahler measure of a complex polynomial in $d$ variables, $p(x)$, is defined to be 
$M(p)\eqdef \frac{1}{(2\pi)^d}\int_{[0,2\pi]^d}\log\parenv{\abs{p(e^{i\theta_1},\dots,e^{i\theta_d})}} d\bar{\theta}$. In the cases which we studied, we showed that the entropy of restricted permutations is given by Mahler measures of polynomials. Is there a more general connection between Mahler measures and the entropy of restricted $\Z^d$ permutations? Given a finite $A\subseteq \Z^d$, is there a canonical way to find a polynomial $P_A$ whose Mahler measure bounds (or is equal to) $\topent(\Omega(A))$?  
\item Periodic entropy and topological entropy: Is it true that $\topent_p(\Omega(A))=\topent(\Omega(A))$ for any finite $A\subseteq \Z^d$?. The answer to that question is positive in the cases studied in this work (where $A=A_+,A_L$).
\item Local and global admissibility: Can Proposition~\ref{prop:compli2} concerning the equivalence between local and global admissibility of rectangular two dimensional patterns be generalized? It seems that it may be true for convex patterns and a wider class of restricting sets. 
\end{enumerate}




\bibliographystyle{AIMS}



\end{document}